\documentclass[]{interact}

\usepackage[dvipsnames]{xcolor}
\usepackage{multirow}
\usepackage{comment}
\usepackage{epstopdf}% To incorporate .eps illustrations using PDFLaTeX, etc.
\usepackage[caption=false]{subfig}% Support for small, `sub' figures and tables
\usepackage{color}

\usepackage[numbers,sort&compress]{natbib}% Citation support using natbib.sty
\bibpunct[, ]{[}{]}{,}{n}{,}{,}% Citation support using natbib.sty
\renewcommand\bibfont{\fontsize{10}{12}\selectfont}% Bibliography support using natbib.sty
\makeatletter% @ becomes a letter
\def\NAT@def@citea{\def\@citea{\NAT@separator}}% Suppress spaces between citations using natbib.sty
\makeatother% @ becomes a symbol again

\theoremstyle{plain}% Theorem-like structures provided by amsthm.sty
\newtheorem{theorem}{Theorem}[section]
\newtheorem{assumption}[theorem]{Assumption}
\newtheorem{lemma}[theorem]{Lemma}

\theoremstyle{definition}

\theoremstyle{remark}

\newcommand{\red}{\textcolor{black}}
\newcommand{\blue}{\textcolor{black}}

\begin{document}

\articletype{ARTICLE TEMPLATE}% Specify the article type or omit as appropriate

\title{Mini-batch stochastic subgradient for functional constrained optimization}

\author{
\name{Nitesh Kumar Singh\textsuperscript{a}\thanks{Corresponding author: Ion Necoara, email: ion.necoara@upb.ro. \\ In this  variant we have corrected some derivations from the  journal version: Optimization, 73(7), 2159–2185, 2023.}, Ion Necoara\textsuperscript{a,b}, Vyacheslav Kungurtsev\textsuperscript{c}}
\affil{\textsuperscript{a}Automatic Control and Systems Engineering Department,  University Politehnica  Bucharest,   Spl. Independentei 313, 060042 Bucharest, Romania; \textsuperscript{b}Gheorghe Mihoc-Caius Iacob Institute of Mathematical Statistics and Applied Mathematics of the  Romanian Academy, 050711 Bucharest, Romania; \textsuperscript{c}Computer Science Department, Czech Technical University, Karlovo Namesti 13, 12135 Prague, Czech Republic. }
}

\maketitle

\begin{abstract}
In this paper we consider finite sum composite \red{convex} optimization problems with  many functional  constraints. The objective function is expressed as  a finite sum of two terms, one of which admits easy computation of  (sub)gradients while the other is amenable to proximal evaluations. We assume a generalized bounded gradient condition on the objective which allows us to simultaneously tackle both smooth and nonsmooth problems. We also consider the cases of both with and without a \red{strong convexity} property. Further, we  assume that each constraint set is  given as the level set of a convex but not necessarily differentiable function.  We reformulate the constrained finite sum problem into a  stochastic optimization problem  for which the stochastic subgradient projection method  from \cite{NecSin:21} specializes  to a collection of mini-batch variants,  with different mini-batch sizes for the objective function and functional constraints, respectively. More specifically, at each iteration, our algorithm takes a mini-batch stochastic proximal subgradient step aimed at  minimizing the objective function and then a subsequent mini-batch subgradient projection step  minimizing the feasibility violation. By specializing different mini-batching strategies,  we derive exact expressions for the stepsizes as a function of  the mini-batch size and in some cases we also derive insightful stepsize-switching rules which describe when one should switch from a constant to a decreasing stepsize regime.  We  also prove sublinear convergence rates for the mini-batch subgradient projection algorithm which depend explicitly on the mini-batch sizes and on the properties of the objective function.  Numerical results also show a better performance of our mini-batch scheme over its single-batch counterpart.
\end{abstract}


%%%%%%%%%%%%%%%%%%%%%%%%%%%%%%%%%%%%%%%%%%%%%%%%%%%%%%%%%%%%%%%%%%%%%%%%%%%%%%%%%%

\section{Introduction}
In this work we consider the following composite  \red{convex}  optimization problem with many functional constraints:

\begin{equation}
	\label{eq:prob}
	\begin{array}{rl}
		F^* = \min_{x \in \mathcal{Y} \subseteq \mathbb{R}^n} & F(x)  \quad \left(:= \frac{1}{N} \sum_{i=1}^{N}(f_i(x) + g_i(x)) \right)\\
		\text{subject to } & h_j(x) \le 0 \;\;\; \forall j=1:m,
	\end{array}
\end{equation}
where $f_i,\; g_i$ and $h_j$ are  proper lower semi-continuous \red{convex} functions,  $\mathcal{Y}$ is a simple closed convex set and the number of sum-additive objective function components, $N$, and/or the number of constraints, $m$, are assumed to be large.  \blue{Hence, we separate the feasible set in two parts:  one set, $\mathcal{Y}$, admits easy projections and the other part is not easy for projection as it is  described by the level sets of  some convex functions $h_j$'s.}   This model is very general and covers many practical optimization applications,  including   machine learning and statistics \cite{Vap:98,BhaGra:04}, distributed control \cite{NedNec:14}, signal processing \cite{Nec:20,Tib:11}, operations research and finance \cite{RocUry:00}. It can be remarked that more commonly, one sees a single $g$  representing the regularizer on the parameters. However, we are interested in the more general problem as there are also applications where one encounters more $g$'s, such as e.g.,  in Lasso problems with mixed $\ell_1-\ell_2$ regularizers, and  in the case of regularizers with overlapping groups \cite{JacObo:09}. Multiple functional constraints can arise from  multistage stochastic programming with equity constraints~\cite{YinBuy:21}, robust classification \cite{BhaGra:04}, and fairness constraints in machine learning~\cite{ZafVal:19}. In the aforementioned applications the corresponding problems are becoming increasingly large in terms of both the number of variables and the size of training data. The use of regularizers and constraints in a composite objective structure make proximal gradient methods particularly natural for these classes of problems, see, e.g.~\cite{Nec:20,RosVil:14}.  Moreover, when the composite objective function is expressed as a large finite sum of functions, then by computational practical necessity, we may have access only to stochastic estimates via samples of the (sub)gradients, proximal operators or projections. In this setting, the most  popular stochastic methods are the stochastic gradient descent (SGD) \cite{RobMon:51, NemYud:83, HarSin:16, Nec:20}  and the stochastic proximal point (SPP) algorithms \citep{MouBac:11, NemJud:09, Nec:20, PatNec:17, RosVil:14}.  However, in practice it has been noticed that these stochastic  methods converge slowly. To improve the convergence speed, one can use  techniques such as mini-batching \citep{AsiDuc:20,  AngNec:19, PenWan:19, RobRic:19, DucSin:09,RenZho:21}, averaging  \cite{NemJud:09,PolJud:92,YanLin:16} or variance reduction strategies \cite{JohZha:13,LinMai:15,GorRic:20}.  In this work we consider a versatile mini-batching framework for a stochastic subgradient projection method for solving the constrained finite sum problem \eqref{eq:prob}, and demonstrate, theoretically and experimentally, its favorable convergence properties. \\ 

\medskip  

\noindent The papers most related to our work are \cite{AngNec:19,RobRic:19,NecSin:21}. However, the optimization problem, the algorithm  and consequently the convergence analysis are different from the present paper.  In particular,  \cite{AngNec:19}  considers the optimization problem \eqref{eq:prob} with a single  nonsmooth convex function $f$ and  $g_i \equiv 0$ for all $i=1:N$.  Additionally, the objective function $f$ is assumed to be strongly convex.  Under this setting,  \cite{AngNec:19}  proposes   a  stochastic subgradient scheme with mini-batch for constraints and derives a sublinear convergence rate for it, whose proof heavily relies on the strong convexity property  of $f$, bounded subgradients of $f$ assumption, and uniqueness of the optimal solution. In this work,   these conditions do not hold anymore as we consider more general assumptions (i.e.,  smooth/nonsmooth functions $f_i$'s,  objective function $F$ is convex or satisfies a strong convexity condition) and   a more general optimization problem (i.e.,  finite sum composite  objective). Moreover, our mini-batch subgradient method differs from the one in \cite{AngNec:19}:  we consider mini-batching to handle both the objective function and the constraints, while  \cite{AngNec:19} considers only mini-batching for constraints; moreover, the data selection rules used to form the mini-batches  are also different in these two papers. Due to these distinctions, our convergence analysis and  rates are not the same as the ones in \cite{AngNec:19}.  In \cite{RobRic:19} an unconstrained finite sum problem is considered, i.e., in problem \eqref{eq:prob}  $g_i \equiv 0$ for all $i= 1:N$, and $h_j \equiv 0$ for all $j=1:m$, and reformulated as a stochastic optimization problem. This reformulation is then solved using SGD. In  this paper we extend the stochastic reformulation  from \cite{RobRic:19} to the  finite sum composite objective function in \eqref{eq:prob}  and add a new stochastic reformulation for the constraints. Then,  we use the stochastic subgradient projection method from  \cite{NecSin:21}   to solve  the reformulated problem, leading to an array of mini-batch variants depending on the data selection rule used to form mini-batches.  This is the first time such an analysis is performed on the general problem \eqref{eq:prob}, and most of our mini-batch variants of the stochastic subgradient projection method are new.\\

\medskip 

\noindent   \textbf{Contributions}.   In this paper we propose   mini-batch variants of  stochastic subgradient projection algorithm for solving the constrained finite sum composite \red{convex} problem \eqref{eq:prob}.  The main advantage of our formulation  is that  the theoretical convergence guarantees of the corresponding numerical scheme only require very basic properties of our problem functions (convexity, bounded gradient type conditions) and  access only to  stochastic  (sub)gradients and  proximal operators. The main  contributions  are:

\medskip 

\noindent \textit{Stochastic reformulation}:  we propose an equivalent   stochastic reformulation  for  the  finite sum composite objective function  and for the constraints of problem  \eqref{eq:prob} using  arbitrary sampling rules. We also extend the assumptions considered for the original problem  to the new stochastic reformulation and derive explicit bounds  for the corresponding constants appearing in the assumptions, which  depend  on the random variables that define the stochastic problem.  By specializing our bounds to different mini-batching strategies, such as partition sampling and nice sampling, we derive exact expressions for these constants. 

\medskip 

\noindent \textit{Convergence rates}:  the stochastic problem  is then solved with the  stochastic subgradient projection method  from \cite{NecSin:21}, which  specializes to a range of possible mini-batch schemes with different batch sizes for the objective function and functional constraints.   Based on the constants defining the assumptions, we derive exact expressions for the stepsize as a function of the mini-batch size. Moreover, when the objective function satisfies a strong convexity  condition we also derive informative stepsize-switching rules which describe when one should switch from a constant to a decreasing stepsize regime. At each iteration, the algorithm takes a mini-batch stochastic proximal subgradient step aimed at minimizing the objective function, followed by a feasibility step  for  minimizing the feasibility violation of the observed mini-batch of random constraints. We prove sublinear convergence rates  for a weighted averages of the iterates  in terms of expected distance to the constraint set, as well as for  expected optimality of the function values/distance to the optimal set. Our rates depend explicitly on the mini-batch sizes and on the properties of the problem functions.  This work is the first analysis of a mini-batch stochastic subgradient projection method on the general problem \eqref{eq:prob}, and most of our mini-batch variants  were never explicitly considered in the literature~before. \\

%\medskip 

\noindent   \textbf{Content}.  In Section 2 we introduce some basic notation and present the main assumptions.  In Section 3 we provide a stochastic reformulation for the original  problem, present several sampling strategies and derive some relevant bounds. In Section 4 we present a mini-batch stochastic subgradient projection algorithm and analyze its convergence. Finally, in Section 5, the performance on numerical simulations is presented,  providing support for the effectiveness of our method.

%%%%%%%%%%%%%%%%%%%%%%%%%%%%%%%%%%%%%%%%%%%%%%%%%%%%%%%%%%%%%%%%%%%%%%%%%%

\section{Notations and assumptions}
For the finite sum problem \eqref{eq:prob} we assume that   $\mathcal{Y}$ is a simple convex set, i.e., it is  easy to evaluate the projection onto $\mathcal{Y}$. Moreover, \blue{we assume that the effective domains of the functions $f_i,\; g_i$ and $h_j$ contain the interior of $\mathcal{Y}$.} Additionally, all the functions $g_i$ have the common domain,  $\text{dom} \, g$. We make no assumptions on the differentiability of $f_i$ and use,  with some abuse of notation, the same expression for the gradient or the subgradient of $f_i$ at $x$, that is $\nabla f_i(x) \in \partial f_i(x)$, where the subdifferential $\partial f_i(x)$ is either a singleton or a nonempty  set for any $i = 1:N$. Similarly for $g_i$'s. Throughout the paper, the subgradient of $h(x, \xi)$ w.r.t. $x$,  $\nabla_x h(x,\xi)$, is denoted simply by $\nabla h(x,\xi)$.  Let us denote  $F_i(x) = f_i(x) + g_i(x)$. Assuming $g_i$'s are convex functions, then from basic calculus rules we have $\nabla F_i(x) = \nabla f_i(x) + \nabla g_i(x)$. Further, for a given $x \in \mathbb{R}^N$,  $\|x\|$ denotes its Euclidean norm and $(x)_+ = \max \{0,x\}$. The feasible set of \eqref{eq:prob} is denoted by:
\[\mathcal{X}=\left\{ x \in  \mathcal{Y}: \;  h_j(x)\le 0 \;\; \forall j = 1:m \right\}. \]
We assume the optimal value  $F^* > -\infty$ and  $\mathcal{X}^* \neq \phi$ denotes  the optimal set, i.e.:
\[F^*=\min_{x  \in \mathcal{X}} F(x) := \frac{1}{N} \sum_{i=1}^{N}F_i(x), \quad \mathcal{X}^*=\{x\in \mathcal{X} \mid F(x)=F^*\}.\] 
For any $x \in \mathbb{R}^n$ we denote its projection onto the optimal set  $\mathcal{X}^*$ by $\bar{x}$, that is:  
$$\bar{x} = \Pi_{\mathcal{X}^*}(x).$$

%\textcolor{red}{There is a small ambiguity: does ~\eqref{as:main1_spg} holds for all $x\in\mathcal{Y}$ or all $x\in\mathcal{X}$? If the latter then this may be an issue in the proofs, since the iterates are not feasible. However, if it (stays) the former then you cannot use the optimality condition in the description of Example 2, since that $\langle F(\bar{x}),x-\bar{x}\rangle\ge 0$ only holds for feasible $x$, i.e., $x\in\mathcal{X}$}

\noindent   We consider additionally the following assumptions. First, we assume that the objective function satisfies some bounded gradient condition.
\begin{assumption}
	\label{assumption1}
	The (sub)gradients of $F$ satisfy the following bounded gradient condition: there exist nonnegative constants  $L \geq 0$ and $B \geq 0$ such that:
	\begin{equation}
		\label{as:main1_spg} B^2 +  L(F(x) - F^*) \geq  \frac{1}{N} \sum_{i=1}^N \| \nabla F_i(x) \|^2  \quad  \forall x \in  \mathcal{Y}.
	\end{equation}
\end{assumption}

\noindent To the best of our knowledge this assumption was first introduced in \cite{Nec:20} and further studied in \cite{NecSin:21,RobRic:19}. We present two examples of functions satisfying this assumption below (see \cite{NecSin:21} for proofs).\\

\noindent \textbf{Example 1} [Non-smooth (Lipschitz) functions satisfy Assumption \ref{assumption1}]:
Assume that the convex functions $f_i$ and $g_i$ have bounded
(sub)gradients:
\[ \|\nabla f_i (x)\| \leq B_{f_i}  \quad \text{and} \quad
\|\nabla g_i(x) \| \leq B_{g_i} \quad \forall x  \in \mathcal{Y}.  \]
Then, Assumption \ref{assumption1} holds with $ L=0 \;\;  \text{and} \;\; B^2 = \frac{2}{N} \sum_{i=1}^N (B_{f_i}^2 + B_{g_i}^2).$\\

\noindent \textbf{Example 2} [Smooth (Lipschitz gradient) functions satisfy Assumption \ref{assumption1}]:  Condition  \eqref{as:main1_spg}  contains the class of convex functions formed as a sum of two \red{convex} terms, one   having Lipschitz continuous gradients with constants $L_{f_i}$'s  and the other  having bounded subgradients over bounded set $\mathcal{Y}$ with constant $B_g$. 
Then, Assumption \ref{assumption1} holds with (here $D$ denotes the diameter of $\mathcal{Y}$):
\[  L = 4 \max_{i=1:N} L_{f_i} \; \text{and} \; B^2 \!=\!  \frac{4}{N} \sum_{i=1}^N B_{g}^2 + 4 \max_{\bar{x} \in \mathcal{X}^*} \left( \!\frac{1}{N} \sum_{i=1}^N \|\nabla f_i(\bar{x})\|^2 + D \max_{i=1:N} L_{f_i} \|\nabla F(\bar{x})\| \!\right)\!. \]

\noindent \red{In our analysis below we also assume $F$ to satisfy a (strong) convexity condition: 
	\begin{assumption}
		\label{assumption2}
		The function  $F$ satisfies a \red{(strong) convex} condition on $\mathcal{Y}$, i.e., there exists non-negative constant $\mu \geq 0$ such that:
		\begin{equation}
			\label{as:strong_spg} 
			F(y) \geq F(x) + \langle \nabla F(x), y-x \rangle +  \frac{\mu}{2}  \|y  - x\|^2 \quad \forall x, y \in \mathcal{Y}.
		\end{equation}
	\end{assumption}
\noindent  Note that when $\mu=0$ relation 	\eqref{as:strong_spg} states that $F$ is convex on $\mathcal{Y}$.}  Additionally, we assume the following bound for the functional constraints:
\begin{assumption}
	\label{assumption3}
	The functional constraints $h_j$  have  bounded subgradients on $\blue{\mathcal{Y}}$, i.e., there exists $B_j>0$ such that:
	\begin{equation}\label{ass:3}
		\|  \nabla h_j(x) \| \le B_j \quad  \forall \, \nabla h_j(x)  \in \partial h_j(x),  \; x \in \blue{\mathcal{Y}}, \;\;   j = 1:m. 
	\end{equation}
\end{assumption}
\noindent Note that this assumption implies that the functional constraints $h_j$ are Lipschitz continuous. Additionally, we assume a H\"{o}lderian growth condition  for the constraints.

\begin{assumption}
	\label{assumption4}
	The functional constraints satisfy additionally the following H\"{o}lderian growth condition for some constants $\bar{c}>0$ and $q\ge 1$:
	\begin{equation}
		\label{eq:constrainterrbound}
		\emph{dist}^{2q}(y, \mathcal{X}) \le \bar{c}  \left( \max_{ j=1:m} ( h_j (y))\right)_+^2   \;\;\; \forall y \in \blue{\mathcal{Y}}.  
	\end{equation}
\end{assumption} 
\noindent Note that this assumption has been used in \cite{RenZho:21} in the context of convex feasibility problems and  in \cite{NecSin:21} for $q=1$ in the context of stochastic optimization problems.  It holds e.g., when the feasible set $\mathcal{X}$ has an interior point, see e.g.  \citep{LewPan:98}, or when the feasible set is polyhedral. However, Assumption \ref{assumption4} holds for more general sets, e.g., when a strengthened Slater condition holds for the collection of functional constraints, such as the generalized Robinson condition, as detailed in \citep{LewPan:98} Corollary 3.

%%%%%%%%%%%%%%%%%%%%%%%%%%%%%%%%%%%%%%%%%%%%%%%%%%%%%%%%%%%%%%%%%%%%%

\section{Stochastic reformulation}
\noindent In this section we reformulate the deterministic problem \eqref{eq:prob} into a stochastic one wherein the objective function is expressed in the form of an expectation. We analyze its main properties  and then use the machinery of stochastic sampling to devise efficient mini-batch schemes. For this we use an arbitrary sampling paradigm. More precisely,  let $(\Omega_1, \mathcal{F}_1, \mathbb{P}_1)$ be a finite probability space with $\Omega_1 = \{ 1,...,N \}$ and a random vector $\zeta \in \mathbb{R}^N$ drawn from some probability distribution $\mathbb{P}_1$ having the property $\mathbb{E} [\zeta^i] = 1 \; \text{for all} \; i =1:N$. Then, let us define the following functions:
\begin{align}\label{eq:reformulation_f}
	f(x, \zeta) = \frac{1}{N} \sum_{i=1}^{N} \zeta^i f_i(x)\;\; \text{and}\;\; g(x, \zeta) = \frac{1}{N} \sum_{i=1}^{N} \zeta^i g_i(x).
\end{align}

\noindent Note that if $f_i$ and $g_i$, with $i=1:N$, are convex functions and $\zeta_i \ge 0$, then $f(\cdot,\zeta)$ and $g(\cdot, \zeta)$ are also convex functions. Also consider a probability space $(\Omega_2, \mathcal{F}_2, \mathbb{P}_2)$, with $\Omega_2 = \{ 1,...,m \}$ and a  random vector $\xi \in \mathbb{R}^m$ drawn from some probability distribution $\mathbb{P}_2$ having the property $\mathbb{E} [\xi^j] > 0$ and $0 \leq \xi^j \le \bar{\xi}$, for all $ j =1:m$ and some $\bar{\xi} < \infty$. Then, let us define the functional constraints:
\begin{align}\label{eq:reformulation_h}
	h(x, \xi) = \max_{j = 1:m} (\xi^j h_j(x)) .
\end{align}

\noindent Since $\xi^j \geq 0$, then $h(\cdot, \xi)$ is a convex function provided that $h_j, \;\text{with}\; j =1:m$, are convex functions. Then, we can define a stochastic reformulation of the original optimization problem~\eqref{eq:prob}:
\begin{equation}
	\label{eq:stochasticProb}
	\begin{array}{rl}
		F^* = & \min\limits_{x \in \mathcal{Y} \subseteq \mathbb{R}^n} \;  \mathbb{E}[f(x,\zeta) + g(x,\zeta)]\\
		& \text{subject to } \;  h(x,\xi) \le 0 \;\; \forall \xi \in \mathcal{F}_2.
	\end{array}
\end{equation}

\noindent Note that $F(x, \zeta) = f(x, \zeta) + g(x, \zeta)$ and $\nabla F(x, \zeta) = \nabla f(x, \zeta) + \nabla g(x, \zeta)$ are unbiased estimators of $F(x)$ and $\nabla F(x)$, respectively. Indeed:
\begin{align*}
	\mathbb{E} [\nabla F(x, \zeta)] \overset{\eqref{eq:reformulation_f}}{=} \frac{1}{N} \sum_{i=1}^{N} \mathbb{E}[\zeta^i] (\nabla f_i(x) + \nabla g_i(x)) \overset{\mathbb{E} [\zeta^i] = 1}{=} \nabla F(x). 
\end{align*}

\noindent In the following lemma we prove that under some basic conditions on the random vectors, the deterministic problem \eqref{eq:prob} is equivalent to the stochastic problem \eqref{eq:stochasticProb}.
\begin{lemma}
	Let the random vectors $\zeta$ and $\xi$ satisfy $\mathbb{E} [\zeta^i] = 1 \; \; \text{for all}\;\; i =1:N$ and $\xi \ge 0$, with $\mathbb{E} [\xi^j] > 0 \; \; \text{for all}\;\; j =1:m$. Then, the the deterministic problem \eqref{eq:prob} is equivalent to stochastic problem \eqref{eq:stochasticProb}.
\end{lemma}

\begin{proof}
	For the objective function in problem \eqref{eq:stochasticProb}, we have:
	\begin{align*}
		\mathbb{E}[f(x,\zeta) + g(x,\zeta)] & = \mathbb{E}\left[\frac{1}{N} \sum_{i=1}^{N} \zeta^i (f_i(x) + g_i(x)) \right]\\
		& = \frac{1}{N} \sum_{i=1}^{N} \mathbb{E} [\zeta^i] (f_i(x) + g_i(x)) = \frac{1}{N} \sum_{i=1}^{N}  (f_i(x) + g_i(x)) \\
		& = F(x).
	\end{align*}  
	For the functional constraints, if $x$ is feasible for the stochastic problem \eqref{eq:stochasticProb}, i.e. $h(x, \xi) \le 0$, then we have:
	\begin{align*}
		h(x, \xi)  &=  \max_{j=1:m} (\xi^j h_j(x))  \le 0  \implies  \xi^j h_j(x)  \le 0 \;\;\; \forall j=1:m.
	\end{align*}
	Taking expectation on both sides, we get:
	\begin{align*}
		\mathbb{E} [\xi^j  h_j(x)] \le 0 \overset{\mathbb{E} [\xi^j] > 0}{\implies}  h_j(x) \le 0 \;\;\; \forall j=1:m,
	\end{align*}
$x$ is feasible for the original problem \eqref{eq:prob}. On the other hand, if $h_j(x) \le 0$, for all $j = 1:m$, then using $\xi^j \ge 0$, we get:
	\begin{align*}
		\xi^j h_j(x) \le 0  \implies \max_{j=1:m} ( \xi^j h_j(x)) \le 0 \implies h(x, \xi) \le 0.
	\end{align*}  
	This concludes our proof.
\end{proof}

%%%%%%%%%%%%%%%%%%%%%%%%%%%%%%%%%%%%%%%%%%%%%%%%

\subsection{Properties of  stochastic problem}
\noindent In this section we prove that the assumptions valid for the original problem \eqref{eq:prob} can be extended  to the stochastic reformulation \eqref{eq:stochasticProb}. Moreover, we derive explicit bounds  for the corresponding assumptions' constants depending  on the random variables that define the stochastic problem.  Let $\nabla \hat{F} (x) $ be the matrix of dimension $n\times N$ obtained by  arranging  $\nabla F_i (x)$'s as its columns. In the next lemma we prove that a stochastic bounded gradient type condition holds for the objective function of  problem~\eqref{eq:stochasticProb}.
\begin{lemma}\label{bddsubgrad_F}
	Let Assumption \ref{assumption1} hold and consider the random vector $\zeta$ satisfying $\mathbb{E} [\zeta^i] = 1$. Then, the (sub)gradients of $F(\cdot, \zeta)$ from the problem \eqref{eq:stochasticProb} satisfy a stochastic bounded gradient condition: 
	\begin{equation} \label{eq:bddsubgrad_F}
		\mathcal{B}^2 +  \mathcal{L} (F(x) - F^*) \geq \mathbb{E}_{\zeta} [\| \nabla F(x, \zeta) \|^2 ] \quad  \forall x \in  \mathcal{Y},
	\end{equation}
	with the parameters $\mathcal{B}^2 = \frac{\mathbb{E}[\|\zeta\|^2]}{N} B^2$ and $\mathcal{L} = \frac{\mathbb{E}[\|\zeta\|^2 ]}{N} L$.
\end{lemma}

\begin{proof} Using the definition of $F(x,\zeta)$, we get: 
	
	\begin{align*} 
		\| \nabla F(x, \zeta) \|^2 & = \left\| \frac{1}{N} \sum_{i=1}^{N} \zeta^i \nabla F_i(x)  \right\|^2 = \frac{1}{N^2} \|\nabla \hat{F} (x) \zeta  \|^2 \le \frac{1}{N^2} \|\nabla \hat{F} (x)\|^2 \| \zeta  \|^2\\
		& \le \frac{1}{N^2} \|\nabla \hat{F} (x)\|_F^2 \| \zeta  \|^2 = \frac{\| \zeta  \|^2}{N} \left(\frac{1}{N} \sum_{i=1}^{N} \|\nabla F_i(x)\|^2\right) \\
		& \overset{\eqref{as:main1_spg}}{\le} \frac{\| \zeta  \|^2}{N} B^2 + \frac{\| \zeta  \|^2}{N} L(F(x) - F(\bar{x})),
	\end{align*}
	where the second inequality follows from the fact that the Frobenius norm is larger than the 2-norm of a matrix. Then, the statement  follows after taking expectation with respect to $\zeta$.
\end{proof}

\noindent \red{From Jensen's inequality, taking $x=x^* \in \mathcal{X}^*$ in  \eqref{as:main1_spg}, we get:
	\begin{equation}
		\label{as:main1_spg2} 
		B^2  \geq \mathbb{E}_\zeta[\| \nabla F(x^*, \zeta) \|^2 ] \ge \| \mathbb{E}_\zeta[\nabla F(x^*, \zeta)] \|^2 = \| \nabla F(x^*)\|^2 \quad \forall  x^* \in \mathcal{X}^*.
	\end{equation}
}

\noindent  Since $F(x,\zeta)$ is an unbiased estimator of $F(x)$, it also follows that if Assumption \ref{assumption2} holds for the original objective function, then the same condition is valid for the objective function of the stochastic problem \eqref{eq:stochasticProb} with the same constant $\mu$. Further, for a given $x$ let us define the set of active constraints by $J^*(x)=\{ j=1:m \; | \;  h(x, \xi) = \xi^j h_j(x)\}$. In the next lemma we provide a bounded subgradient  condition for the functional constraints of the stochastic problem \eqref{eq:stochasticProb}.
\begin{lemma}\label{bddsubgrad_h}
	Let Assumption \ref{assumption3} hold and consider the random vector $\xi \geq 0$ satisfying $ \mathbb{E}[\xi^j] > 0$ and $ \xi^{j} \le \bar{\xi}$, for all  $j=1:m$ and some $\bar{\xi} < \infty$. Then, the functional constraints $h(\cdot,\xi)$ of the problem \eqref{eq:stochasticProb} have  bounded subgradients on $\blue{\mathcal{Y}}$, i.e.: 
	\begin{align}\label{eq:bddsubgrad_h}
		\|  \nabla h(x,\xi) \| \le \mathcal{B}_h \quad   \forall  x  \in    \blue{\mathcal{Y}}  \;\;  \text{and } \;\;  \xi \in \mathcal{F}_2,
	\end{align}   
	where $\nabla h(x,\xi)  \in \partial h(x,\xi)$ and $ \mathcal{B}_h=\bar{\xi} \max_{ j=1:m} B_j. $
\end{lemma}

\begin{proof}
	Let $x \in \blue{\mathcal{Y}}$ and $\nabla h(x, \xi) \in \partial h(x, \xi)$. Then, from the definition of $h(\cdot, \xi)$ and of the index set $ J^*(x)$, we have:
	\[ h(x,\xi) =  \max_{j=1:m} (\xi^j h_j(x)) = \xi^{j^*}\cdot h_{j^*}(x) \quad \forall j^* \in J^*(x). \]
	Then, we further have (see Lemma 3.1.13 in \cite{Nes:18}):
	\begin{align}\label{eq:subgradh}  
	    &\nabla h(x,\xi) = \text{Conv} \{ \xi_{j^*} \nabla h_{j^*} (x) |j^* \in J^*(x) \} \nonumber\\
		 \implies & \|\nabla h(x, \xi)\| \le \max_{\theta_{j^*}\ge 0,\; \sum_{j^*\in J^*(x)} \theta_{j^*}=1 } \left\|\sum_{j^*\in J^*(x)} \theta_{j^*}\xi^{j^*} \cdot \nabla h_{j^*}(x) \right\| \nonumber \\
		& \le \max_{\theta_{j^*}\ge 0,\; \sum_{j^*\in J^*(x)} \theta_{j^*}=1 } \sum_{j^*\in J^*(x)} \theta_{j^*}\xi^{j^*}\cdot \|\nabla h_{j^*}(x)\| \nonumber \\ 
		& \le \max_{\theta_{j^*}\ge 0,\; \sum_{j^*\in J^*} \theta_{j^*}=1 } \sum_{j^*\in J^*(x)} \theta_{j^*} \bar{\xi}\cdot \|\nabla h_{j^*}(x)\| \nonumber \\
		& = \bar{\xi} \max_{j^* \in J^*(x)} \|\nabla h_{j^*}(x) \| \overset{\eqref{ass:3}}{\le} \bar{\xi} \max_{ j=1:m} B_j =\mathcal{B}_h,
	\end{align} 
which proves our statement.	
\end{proof}

\noindent In the next lemma we provide a H\"{o}lderian growth type condition for the functional constraints of the stochastic problem \eqref{eq:stochasticProb}.
\begin{lemma}\label{linearreg_h}
	Let Assumption \ref{assumption4} hold and consider the random vector $\xi$ satisfying $\xi \ge 0$ and $\mathbb{E}[\xi^j] > 0$ for all  $j=1:m$. Then, the functional constraints of the problem \eqref{eq:stochasticProb} satisfy the following H\"{o}lderian growth type condition:
	\begin{equation}\label{qreg}
		\emph{dist}^{2q}(y, \mathcal{X}) \le c \cdot  \mathbb{E} \left[ (h(y,\xi))_+^2 \right] \;\; \forall y \in \blue{\mathcal{Y}}, 
	\end{equation}
with the parameter  $c = \left(\frac{\bar{c}}{\min_{j = 1:m} \mathbb{E}[\xi^j]} \right)$.
\end{lemma}

\begin{proof}
	Let $y \in \blue{\mathcal{Y}}$, using the definition of $h(\cdot, \xi)$ and Jensen's inequality, we have: 
	\begin{align*}
		\mathbb{E} \left[ (h(y,\xi))_+^2 \right] & = \mathbb{E} \left[ \left(\max_{j = 1:m}(\xi^j h_j (y))\right)_+^2 \right]  \ge \left( \max_{j = 1:m}( \mathbb{E} [\xi^j] h_j (y))  \right)_+^2 \\
		& \overset{\mathbb{E}[\xi^j] > 0}{\ge} \min_{j = 1:m} \mathbb{E}[\xi^j] \left( \max_{j = 1:m}( h_j (y)) \right)_+^2 \overset{\eqref{eq:constrainterrbound}}{\ge} \left(\min_{j = 1:m} \mathbb{E}[\xi^j]\right) \frac{1}{\bar{c}} \text{dist}^{2q}(y, \mathcal{X}).
	\end{align*}
	Thus, we have:
	\begin{align}\label{eq:linearreg}
		\text{dist}^{2q}(y, \mathcal{X}) & \le \frac{\bar{c}}{\min\limits_{j = 1:m} \mathbb{E}[\xi^j]} \mathbb{E} \left[ (h(y,\xi))_+^2 \right],
	\end{align}
which proves our statement.	
\end{proof}

%%%%%%%%%%%%%%%%%%%%%%%%%%%%%%%%%%%%%%%%%%%%%%%%%%%%%%%%%%%%%%%%%%%%%%%%%%%%%%

\subsection{Choices for random vectors $\zeta$ and $\xi$}

In this section, we provide several choices for the two random vectors $\zeta$ and $\xi$. Let $\mathcal{I} \subseteq [1:N]$ and let $e_\mathcal{I} = \sum_{i \in \mathcal{I}} e_i$, where $\{e_1, ...,e_N\}$ is the standard basis of $\mathbb{R}^N$. These subsets will be selected using a random set valued map, i.e. sampling $S$. A sampling $S$ is uniquely characterized by choosing the probabilities $p_\mathcal{I} \geq 0$ for all subsets $\mathcal{I}$:
\[ \mathbb{P} [S = \mathcal{I}] = p_\mathcal{I} \;\;\forall \mathcal{I}\subset [1:N],  \]
such that $\sum_{\mathcal{I}\subset [1:N]} p_\mathcal{I} = 1$. A sampling $S$ is called proper if $p_i = \mathbb{P} [i \in S] = \mathbb{E} [\mathbf{1}_{i\in S}]= \sum_{\mathcal{I}:i \in \mathcal{I}} p_\mathcal{I}$ is positive for all $i=1:N$, see also \cite{RobRic:19, RicTak:16}. We now define some practical sampling vectors $\zeta = \zeta(S)$. For example, let $S$ be a proper sampling and let $\hat{\mathbb{P}} = \text{Diag}(p_1,...,p_N)$. Then, we can consider the sampling vector as:
\begin{equation} \label{samplingVector1}
	\zeta = \hat{\mathbb{P}}^{-1}e_S \implies \zeta^i = \frac{\mathbf{1}_{i \in S}}{p_i}.
\end{equation}

\noindent Note that $\mathbb{E} [\zeta^i] = \frac{\mathbb{E} [\mathbf{1}_{i \in S}]}{p_i} = 1$ and since $\zeta^T \zeta = \sum_{i=1}^{N} (\zeta^i)^2 = \sum_{i=1}^{N} \mathbf{1}_{i \in S}/ p_i^2$, then  $\mathbb{E} [\|\zeta\|^2] = \sum_{i=1}^{N} 1/p_i$. For constraints, if we let $\mathcal{I}' \subseteq [1:m]$ and define $e_{\mathcal{I}'} = \sum_{j \in \mathcal{I}'} e_j$, then a sampling $S'$ is uniquely characterized by choosing probabilities $p_{\mathcal{I}'} \geq 0$ for all subsets $\mathcal{I}'$ of $[1:m]$. Let $S'$ be a proper sampling vector, then we can define the practical sampling vector $\xi = \xi(S')$ as:
\begin{equation} \label{samplingVector2}
	\xi =  e_{S'} \implies \xi^j = \mathbf{1}_{j \in S'}.
\end{equation}
Note that $\mathbb{E} [\xi^j] = \mathbb{E} [\mathbf{1}_{j \in S'}] = p_j > 0$ and $ \xi^j \leq \bar{\xi} = 1 $. Furthermore, each sampling $S$ and $S'$ give rise to a particular sampling vector $\zeta = \zeta(S)$ and $\xi = \xi (S')$. Below we provide some sampling examples.\\

\noindent\textit{Partition sampling:} A partition $\mathcal{P}$ of $[1:N]$ is a set consisting of subsets of $[1:N]$ such that $\cup_{\mathcal{I}\in \mathcal{P}} \mathcal{I} =[1:N]$ and $\mathcal{I}_i \cap \mathcal{I}_l = \phi$ for any $\mathcal{I}_i, \mathcal{I}_l \in \mathcal{P}$ with $i \ne l$. A partition sampling $S$ is a sampling such that $p_\mathcal{I} = \mathbb{P}[S=\mathcal{I}] > 0$ for all $\mathcal{I} \in \mathcal{P}$ and $\sum_{\mathcal{I} \in \mathcal{P}} p_\mathcal{I} = 1$. \\

\noindent\textit{$\tau$-nice sampling:} We say that $S$ is $\tau$–nice if $S$ samples from all subsets of $[1:N]$ of cardinality $\tau$ uniformly at random. In this case we have that $p_i = \frac{\tau}{N} \; \text{for all}\; i = 1:N$. Then, $p_\mathcal{I} = \mathbb{P}\left[S =\mathcal{I} \right] = 1/{N \choose \tau}$ for all subsets $\mathcal{I}\subset\{1,...,N\}$ with $\tau$ elements.\\

\noindent The reader can also consider  other examples for sampling, see e.g., \cite{RobRic:19} for more details. 
%Let us introduce the matrix $\mathbf{P}\in \mathbb{R}^{N\times N}$ defined as $\mathbf{P}_{il} = \mathbf{P}[i\in S, l \in S]$.
Let the cardinality of samples $S$ and $S'$ be $\tau_1$ and $\tau_2$, respectively. In the next theorem, using Lemmas \ref{bddsubgrad_F}, \ref{bddsubgrad_h} and \ref{linearreg_h}, we derive explicit expressions. which depend on the mini-batch sizes $\tau_1$ and $\tau_2$, for the assumptions' constants  $\mathcal{B}, \mathcal{L}, \mathcal{B}_h$ and $c$ for the two sampling  given previously.

\begin{theorem} \label{sampling_cases}
	Let Assumption \ref{assumption1}, \ref{assumption3} and \ref{assumption4} hold. Let also $S$ and $S'$ be sampled uniform at random with partition sampling  having the same cardinality $\tau_1$ and $\tau_2$, or alternatively with $\tau_1$- and $\tau_2$-nice sampling. Then, the constants $\mathcal{B}, \mathcal{L}, \mathcal{B}_h$ and $c$ are:
	\[\mathcal{B}^2 = \frac{N}{\tau_1}B^2,\; \mathcal{L} = \frac{N}{\tau_1}L,\; \mathcal{B}_h = \max_{j=1:m} B_j\;  \text{and}\; c = \left(\frac{\bar{c}m}{\tau_2} \right).\]
	%\[\mathcal{B}^2 = B^2\tau_1,\; \mathcal{L} = L \tau_1,\; \mathcal{B}_h = \max_{j=1:m} B_j\; \text{and} \; c =  \left(\frac{\bar{c}m}{\tau_2} \right) .\]
	%(ii) Let $S$ and $S'$ sampled with $\tau_1$ and $\tau_2$-nice sampling, respectively, then the constants $\mathcal{B}, \mathcal{L}, \mathcal{B}_h$ and $c$ are defined as:
\end{theorem}
\begin{proof}
From Lemma \ref{bddsubgrad_F}, for the parameters $\mathcal{B}\;\text{and}\; \mathcal{L}$, we have:
    \begin{align}\label{eq:common2}
        & \mathcal{B}^2 = \frac{\mathbb{E}[\|\zeta\|^2]}{N} B^2  \overset{\eqref{samplingVector1}}{=} \frac{1}{N}\sum_{S} p_S\sum_{i\in S} \frac{1}{p_i^2} B^2,\\ 
        & \mathcal{L} = \frac{\mathbb{E}[\|\zeta\|^2 ]}{N} L  \overset{\eqref{samplingVector1}}{=} \frac{1}{N}\sum_{S} p_S\sum_{i\in S} \frac{1}{p_i^2} L.\nonumber
    \end{align}
   % Using the definition of $\zeta$ from \eqref{samplingVector1}, we get:
    %\begin{align}\label{eq:common2}
     %   \mathcal{B}^2 = \frac{1}{N} \sum_{i=1}^{N} \frac{1}{p_i} B^2\;\; \text{and}\;\; \mathcal{L} = \frac{1}{N} \sum_{i=1}^{N} \frac{1}{p_i} L.
    %\end{align}
    For partition sampling given the realization $S = \mathcal{I}$, we have $p_{i} = p_\mathcal{I}$ if $i \in \mathcal{I}$. Since the cardinality of each $\mathcal{I}$ is $\tau_1$ and the sampling $S \in \{ \mathcal{I}_1,..., \mathcal{I}_{\ell} \}$ is chosen uniform at random, then $p_i = p_\mathcal{I} = \frac{1}{\ell} = \frac{\tau_1}{N}$. Thus, using \eqref{eq:common2} we have:
    \begin{align*}
        \mathcal{B}^2 = \frac{1}{N}\sum_{\mathcal{I} \in \mathcal{P}} p_\mathcal{I} \sum_{i\in \mathcal{I}} \frac{1}{p_i^2} B^2 = \frac{1}{N}\sum_{\mathcal{I} \in \mathcal{P}} \frac{\tau_1}{N} \tau_1 \frac{N^2}{\tau_1^2} B^2 = \sum_{\mathcal{I} \in \mathcal{P}} B^2 = \ell B^2 =  \frac{N}{\tau_1} B^2. 
        %\text{and}\;\; \mathcal{L} =  \frac{1}{N}\sum_{\mathcal{I} \in \mathcal{P}} p_\mathcal{I} \sum_{i\in \mathcal{I}} \frac{1}{p_i^2} L\\
        %& \implies \mathcal{B}^2 = \frac{1}{N}\sum_{\mathcal{I} \in \mathcal{P}} \frac{\tau_1}{N} \tau_1 \frac{N^2}{\tau_1^2} B^2\;\; \text{and}\;\; \mathcal{L} = \frac{1}{N}\sum_{\mathcal{I} \in \mathcal{P}} \frac{\tau_1}{N} \tau_1 \frac{N^2}{\tau_1^2} L\\
        %& \implies \mathcal{B}^2 = \sum_{\mathcal{I} \in \mathcal{P}} B^2 = \ell B^2 =  \frac{N}{\tau_1} B^2\;\; \text{and}\;\; \mathcal{L} = \sum_{\mathcal{I} \in \mathcal{P}}  L = \ell L = \frac{N}{\tau_1} L.
    \end{align*}
    Similarly, we can prove that $\mathcal{L} = \frac{N}{\tau_1} L$. 
    %\[ \mathcal{B}^2 = \frac{1}{N} \sum_{i=1}^{N} \frac{N}{\tau_1} B^2 = \frac{N}{\tau_1} B^2\;\; \text{and}\;\; \mathcal{L} = \frac{1}{N} \sum_{i=1}^{N} \frac{N}{\tau_1} L = \frac{N}{\tau_1} L. \]
    For $\tau_1$-nice sampling given the realization $S = \mathcal{I}$, we have, $p_i = \frac{\tau_1}{N}$ for all $i$ and $p_\mathcal{I} = 1/{N \choose \tau_1}$. Using \eqref{eq:common2}, we get:
    \[ \mathcal{B}^2 = \frac{B^2}{N}\sum_{\mathcal{I}} p_\mathcal{I} \sum_{i\in \mathcal{I}} \frac{1}{p_i^2} = \frac{B^2}{N}\sum_{\mathcal{I}} \frac{1}{{N \choose \tau_1}} \tau_1 \frac{N^2}{\tau_1^2} = \frac{N}{\tau_1} B^2 \sum_{\mathcal{I}}\frac{1}{{N \choose \tau_1}} =  \frac{N}{\tau_1} B^2.\]
    Similarly, we can get the value for other parameter, i.e., $\mathcal{L} = \frac{N}{\tau_1} L$. 
    By Lemma \ref{bddsubgrad_h}, for the parameter $\mathcal{B}_h$, we  have:
    \[ \mathcal{B}_h = \bar{\xi} \max_{ j=1:m} B_j. \]
    Using the definition of $\xi^j$ from \eqref{samplingVector2}, i.e., $ \xi^j \leq \bar{\xi} = 1 $, we get:
	\[ \mathcal{B}_h = \max_{ j=1:m} B_j. \]
	Note that this bound holds for both types of sampling. Finally, from  Lemma \ref{linearreg_h}, for the parameter $c$, we have:
	\[ c = \frac{\bar{c}}{\min_{j = 1:m} \mathbb{E}[\xi^j]} \overset{\eqref{samplingVector2}}{=} \frac{\bar{c}}{\min_{j =1:m} p_{j}} .\]
	Here we use the fact that $ \mathbb{E}[\xi^j] = \mathbb{E} [\mathbf{1}_{j \in S'}] = p_j $.
	Now for the given realization $S'=\mathcal{I}'$, we have $p_j = p_{\mathcal{I}'} = \frac{\tau_2}{m}$ for partition sampling and $p_j = \frac{\tau_2}{m}$ for $\tau_2$-nice sampling, respectively. Therefore,  $c = \left(\frac{\bar{c}m }{\tau_2} \right)$. These prove our statements. 
\end{proof}

\section{Mini-batch stochastic subgradient projection algorithm }

For solving the  stochastic reformulation \eqref{eq:stochasticProb} of the optimization problem \eqref{eq:prob} we adapt the stochastic subgradient projection  method  from \cite{NecSin:21}. We refer to this algorithm as the Mini-batch Stochastic Subgradient Projection  method (\texttt{Mini-batch SSP}). 
\begin{center}
	\noindent\fbox{%
		\parbox{12cm}{%
			\textbf{Algorithm 1 (\texttt{Mini-batch SSP})}:\\
			$\text{Choose} \; x_0 \in \mathcal{Y} \; \text{and stepsizes} \; \alpha_k>0, \; \beta \in (0, 2).$\\
			$\text{For} \; k \geq 0 \;  \text{repeat:}$
			\vspace{-0.4cm}
			\begin{align}
			    & \text{Draw} \; \text{sample vectors} \; \zeta_k \sim \mathbb{P}_1 \;\text{and}\; \xi_k\sim \mathbb{P}_2 \;  \text{independently}.\\
				&\blue{u_k} = \text{prox}_{\alpha_k g(\cdot,\zeta_{k})} \left(x_{k} - \alpha_k \nabla f(x_{k}, \zeta_{k} )\right), \quad v_k = \blue{\Pi_{\mathcal{Y}} (u_k)} \label{eq:algstep1}\\
				&\text{Compute}\; h(v_k, \xi_k) = \max(\xi^1_{k}h_1(v_k),..., \xi^m_{k}h_m(v_k) ) \nonumber \\
				&z_k = v_k  - \beta \frac{(h(v_k,\xi_k))_+}{\| \nabla h(v_k,\xi_k) \|^2} \nabla h(v_k,\xi_k) \label{eq:algstep2}\\
				&x_{k+1} = \Pi_{\mathcal{Y}} (z_k). \nonumber
			\end{align}
		}%
	}
\end{center}

\noindent Using the sampling paradigm in Section 3, the   \texttt{Mini-batch SSP} algorithm  can incorporate a diverse array of mini-batch variants, each of which is associated with a specific probability law governing the data selection rule used to form mini-batches. Most of our variants of \texttt{Mini-batch SSP}, with different mini-batch sizes for the objective function and functional constraints, were never explicitly considered in the literature before, e.g., the variants corresponding to  partition and nice samplings.  Note that at each iteration our algorithm takes a mini-batch stochastic proximal subgradient step aimed at  minimizing the objective function (see \eqref{eq:algstep1}) and then a subsequent mini-batch subgradient projection step  minimizing the feasibility violation (see \eqref{eq:algstep2}). More precisely, if the random vector $\zeta_k$ has  $\zeta_k^i =1$ for all $i \in \mathcal{I}_k$ and $\zeta_k^i =0$ for all $i \in \{1,\cdots, N\} \setminus \mathcal{I}_k$, then step \eqref{eq:algstep1} is a mini-batch proximal subgradient iteration \blue{and followed by a projection step onto the set $\mathcal{Y}$}:
\[  \blue{u_k} = \text{prox}_{\alpha_k \sum_{i \in \mathcal{I}_k} g_i} \left(x_{k} - \alpha_k  \sum_{i \in \mathcal{I}_k} \nabla f_i(x_{k}) \right), \quad v_k  = \blue{\Pi_{\mathcal{Y}} (u_k)}. \]
Similarly, if the random vector $\xi_k$ has  $\xi_k^i =1$ for all $i \in \mathcal{I}_k'$ and $\xi_k^i =0$ for all $i \in \{1,\cdots, m\} \setminus \mathcal{I}_k'$, then step \eqref{eq:algstep2} minimizes the feasibility violation of the observed mini-batch of constraints, i.e., we choose from the mini-batch the constraint that is violated the most,   $h(v_k, \xi_k) = \max_{j \in \mathcal{I}_k'} h_j(v_k) = h_{j_k^*}(v_k)$ for some index $j_k^* \in \mathcal{I}_k'$, and then perform a Polyak's subgradient like update on it \cite{Pol:67}:
\[  z_k = v_k  - \beta \frac{(h(v_k,\xi_k))_+}{\| \nabla h(v_k,\xi_k) \|^2} \nabla h(v_k,\xi_k) = v_k  - \beta \frac{(h_{j_k^*}(v_k))_+}{\| \nabla h_{j_k^*}(v_k) \|^2} \nabla h_{j_k^*}(v_k). \]
Consider any arbitrary nonzero $s_h\in \mathbb{R}^n$. Disregarding the abuse of notation, we compute the vector $\nabla h(v_k, \xi_k) = \nabla h_{j_k^*}(v_k)$ by:
\[
 \nabla h_{j_k^*}(v_k)  = 
\begin{cases} 
	 \nabla h_{j_k^*}(v_k) \in \partial h_{j_k^*}(v_k)  & \mbox{if } \;  h_{j_k^*}(v_k) > 0  \\ 
	s_h  \neq 0 & \mbox{if } \;  h_{j_k^*}(v_k) \leq  0. 
\end{cases}
\] 
When $(h(v_k,\xi_k))_+ = (h_{j_k^*}(v_k))_+ = 0$, we have $z_k = v_k$ for any choice of $s_h \neq 0$.  Note that in the \texttt{Mini-batch SPP}  algorithm  $\alpha_{k}>0$ and $\beta>0$ are deterministic stepsizes. Moreover, when $\beta = 1$, $z_k$ is the projection of  $v_k$ onto the  hyperplane given by the functional constraint that is violated the most in  the observed mini-batch of constraints given by the index set $\mathcal{I}_k'$: 
$$\mathcal{H}_{v_k,\xi_k} =\{ z:   h(v_k,\xi_k) + \nabla h(v_k,\xi_k)^T(z - v_k) \!\leq\! 0  \} \!=\! \{ z:   h_{j_k^*}(v_k) + \nabla h_{j_k^*}(v_k)^T(z - v_k) \!\leq\! 0  \}, $$ that is,  we have  $z_k = \Pi_{\mathcal{H}_{v_k,\xi_k}} (v_k)$ when we choose $\beta = 1$. In the next sections we analyse the convergence behaviour of \texttt{Mini-batch SSP} algorithm  and derive rates depending explicitly on the mini-batch sizes and on the properties of the objective function.

%%%%%%%%%%%%%%%%%%%%%%%%%%%%%%%%%%%%%%%%%%%%%%%%%%%%%%%%%%

\subsection{Convergence analysis: convex objective function}
\noindent In this section we consider that the functions $f_i, g_i$ and  $h_j$ in problem \eqref{eq:prob} are convex and the random vectors $\zeta$ and $\xi$ are non-negative.   Let us define the filtration as the sigma algebra generated by the history of the random vectors $\zeta$ and $\xi$: 
\[  \mathcal{F}_{[k]}= \sigma(\{ \zeta_t, \xi_t: \;   0 \le t \le k \}).   \]
The next lemma, whose proof is similar to Lemma 5 in \cite{NecSin:21} provides a key  descent property for  the sequence $v_k$ (recall that $\bar{v}_{k} = \Pi_{\mathcal{X}^*} (v_k)$ and $\bar{x}_{k} = \Pi_{\mathcal{X}^*} (x_k)$).

\begin{lemma}
	\label{th:spg_basic} 
	Let  $f_i$ and $g_i$, with $i = 1:N$, be  convex functions and $\zeta\ge 0$.  Additionally, let the bounded gradient condition from Assumption \ref{assumption1} hold. Then, for any $k \geq 0$ and  stepsize $\alpha_k >0$, we have the following recursion:
	\begin{align}
		\label{spg_basic} 
		&\mathbb{E}[\|v_{k} - \bar{v}_{k} \|^2] 
		\leq \mathbb{E}[\| x_k - \bar{x}_{k} \|^2] -  \alpha_k  (2 - \alpha_k \mathcal{L}) \, \mathbb{E}[ F(x_k) - F(\bar{x}_{k}) ] + \alpha_k^2 \mathcal{B}^2,   %\nonumber
	\end{align}
	with $\mathcal{B}$ and $\mathcal{L}$  given in Lemma \ref{bddsubgrad_F}.
\end{lemma}

\noindent The following lemma establishes a relation between $x_k$ and $v_{k-1}$. The proof is  similar to Lemma  6 in~\cite{NecSin:21}.  

\begin{lemma}
	\label{lem:distanyy}
	Let  $h_j$, with $j=1:m$, be  convex functions and $\xi\ge 0$.	Additionally, assume that the bounded subgradient condition from Assumption \ref{assumption3} holds.   Then,  for any $y \in  \mathcal{Y}$ such that $(h(y,\xi_{k-1}))_+ = 0$, the following relation holds:
	\begin{align} \label{eq:x_k_v_k-1}
		\|x_{k}-y\|^2\le \|v_{k-1}-y\|^2-\beta(2-\beta) \left[ \frac{(h(v_{k-1},\xi_{k-1}))_+^2}{\mathcal{B}^2_h} \right],
	\end{align}
	with $\mathcal{B}_h$  given in Lemma \ref{bddsubgrad_h}.
\end{lemma}

\noindent Taking now $y=\Pi_{\mathcal X} ({v}_{k-1}) \subseteq \mathcal{X} \subseteq \mathcal{Y}$, then $(h(\Pi_{\mathcal{X}}({v}_{k-1}),\xi_{k-1}))_+ = 0$ and
\begin{align*}  
	\text{dist}^2(x_{k}, \mathcal{X}) & = \|x_{k} - \Pi_{\mathcal X} ({x}_{k})\|^2 \le \|x_{k} - \Pi_{\mathcal X} ({v}_{k-1}) \|^2 \\  
	& \overset{\eqref{eq:x_k_v_k-1}}{\le}  \text{dist}^2(v_{k-1}, \mathcal{X}) -\beta(2-\beta) \frac{(h(v_{k-1},\xi_{k-1}))_+^2}{\mathcal{B}^2_h}\\
	& \le \text{dist}^2(v_{k-1}, \mathcal{X}). 
\end{align*}
Thus for any $q\ge1$, we have:
\begin{align}\label{eq:distvdistx}
	\text{dist}^{2q}(x_{k}, \mathcal{X}) \le \text{dist}^{2q}(v_{k-1}, \mathcal{X}).
\end{align}

\begin{lemma}\label{lem:xkvk-1}
	Let  Assumptions \ref{assumption3} and \ref{assumption4} hold and the random vectors $\xi$ and $\zeta$ be nonnegative. Then, the following relation is valid:
	\[ \mathbb{E}[\|x_{k}-\bar{x}_k\|^2] \le \mathbb{E}[\|v_{k-1}- \bar{v}_{k-1}\|^2] - \frac{\beta(2-\beta)}{c\mathcal{B}^2_h} \mathbb{E}\left[ \emph{dist}^{2q}(x_k, \mathcal{X}) \right], \]
	with $\mathcal{B}_h$ and $c$  given in Lemmas \ref{bddsubgrad_h} and \ref{linearreg_h}, respectively.
\end{lemma}

\begin{proof}
	Note that for $\bar{v}_{k-1} \in \mathcal{X}^* \subseteq \mathcal{X} \subseteq \mathcal{Y}$ we have  $(h(\bar{v}_{k-1},\xi_{k-1}))_+ = 0$ and  using Lemma~\ref{lem:distanyy} with $y=\bar{v}_{k-1}$,  we get: 
	\begin{align*}
		\|x_k - \bar{x}_k\|^2  \le \|x_{k}-\bar{v}_{k-1}\|^2 \leq \|v_{k-1} -\bar{v}_{k-1}\|^2 - \beta(2-\beta) \left[\frac{(h(v_{k-1},\xi_{k-1}))_+^2}{\mathcal{B}^2_h}\right].
	\end{align*}
	Taking conditional expectation on  $\xi_{k-1}$ given  $\mathcal{F}_{[k-2]} $, we get:
	\begin{align*}
		\mathbb{E}_{\xi_{k-1}} [\|x_k - \bar{x}_k\|^2 |  \mathcal{F}_{[k-2]} ] &\leq \|v_{k-1} -\bar{v}_{k-1}\|^2 - \beta(2-\beta) \mathbb{E}_{\xi_{k-1}} \left[\frac{(h(v_{k-1},\xi_{k-1}))_+^2}{\mathcal{B}^2_h}| \mathcal{F}_{[k-2]} \right]\\
		& \overset{\eqref{qreg}}{\leq} \|v_{k-1} -\bar{v}_{k-1}\|^2 - \frac{ \beta(2-\beta)}{c \mathcal{B}^2_h}  \text{dist}^{2q}(v_{k-1}, \mathcal{X})\\
		& \overset{\eqref{eq:distvdistx}}{\leq} \|v_{k-1} -\bar{v}_{k-1}\|^2 - \frac{ \beta(2-\beta)}{c \mathcal{B}^2_h}  \text{dist}^{2q}(x_k, \mathcal{X}).
	\end{align*}
	Taking now the full expectation, we obtain our statement. 
%	\begin{align*}
%		\mathbb{E} [\|x_k - \bar{x}_k\|^2  ] & \leq  \mathbb{E} [ \|v_{k-1} -\bar{v}_{k-1}\|^2 ]  - \frac{ \beta(2-\beta)}{c \mathcal{B}^2_h}  \mathbb{E} [ \text{dist}^{2q}(x_k, \mathcal{X})].
%	\end{align*}			
\end{proof}

\noindent For simplicity of the exposition let us introduce the following  constant: 
\begin{align}\label{C}
 C_{\beta,c,\mathcal{B}_h} := \frac{ \beta(2-\beta)}{c \mathcal{B}^2_h} > 0.  
\end{align}

\red{\noindent We impose the following conditions on the stepsize $\alpha_k$:%\footnote{Journal version from Optimization, 2022, DOI: 10.1080/02331934.2023.2189015,  has some wrong derivations that have been corrected in the current manuscript.}:
	\begin{align}
		\label{eq:alk}
		0 < \alpha_k \leq \alpha_k(2-\alpha_k \mathcal{L}) <1 \;\; \iff \;\; 	 \alpha_k \in 
		\begin{cases}
			\left(0, \frac{1}{2} \right) \;\;  \text{if} \; \mathcal{L} =0  \\
			\left(0, \frac{1- \sqrt{(1-\mathcal{L})_+}}{\mathcal{L}} \right) \;\;  \text{if} \; \mathcal{L} > 0.
		\end{cases} 
	\end{align}	 
}
Then, we can define the following  average sequence generated by the algorithm SSP: 
\[  \hat{x}_k = \frac{\sum_{j=1}^{k} \alpha_j \red{(2 -\alpha_j \mathcal{L})} x_j}{S_k}, \quad  \text{where}  \; S_k = \sum_{j=1}^{k} \alpha_j \red{(2 -\alpha_j \mathcal{L})}.   \] 
\red{Note that this type of average sequence is also consider in \cite{GarGow:23} for unconstrained stochastic optimization problems.}
%Let us also define the average sequence generated by  \texttt{Mini-batch SSP}: 
%\[  \hat{x}_k = \frac{\sum_{t=1}^{k} \alpha_t x_t}{S_k}, \quad  \text{where}  \; S_k = \sum_{t=1}^{k} \alpha_t.   \] 
The next theorem derives \blue{an estimate} for the average sequence  $\hat{x}_k$. 

\begin{theorem}
	\label{th:nonstrongconv}
	Let  $f_i$, $g_i$, with $i=1:N$, and  $h_j$, with $j=1:m$,  be  convex functions. Additionally,   Assumptions \ref{assumption1},  \ref{assumption3} and \ref{assumption4} hold and the random vectors $\zeta,\; \xi$ are nonnegative.  Further, consider a nonincreasing  positive stepsize sequence \red{$\alpha_k$ satisfying \eqref{eq:alk}} and  stepsize  $\beta \in (0, 2)$. Then, we have the following \blue{estimates} for the average sequence  $\hat{x}_k$ in terms of optimality and feasibility violation for problem \eqref{eq:prob}: 
	\begin{align*}
		& \mathbb{E}\left[  F(\hat{x}_k ) - F^*  \right] \leq    \frac{\| v_0 - \overline{v}_{0} \|^2}{ S_k }  +  \frac{\mathcal{B}^2 \sum_{t=1}^{k} \alpha_t^2 }{ S_k }, \\
		& \mathbb{E}\left[ \emph{dist}^{2}(\hat{x}_k, \mathcal{X}) \right] \le \left(\frac{1}{C_{\beta,c,\mathcal{B}_h}\cdot S_k}\right)^{\frac{1}{q}} \left[ \|v_0 - \bar{v}_0\|^{\frac{2}{q}} + \mathcal{B}^{\frac{2}{q}} \sum_{t=1}^{k} \alpha_t^{\frac{2}{q}} \right].
	\end{align*}  
\end{theorem}

\begin{proof}
Combining Lemma \ref{lem:xkvk-1} with Lemma \ref{th:spg_basic}, we have:
	\begin{align*}
		& \mathbb{E}\left[\|v_{k} - \bar{v}_{k} \|^2\right] + \frac{ \beta(2-\beta)}{c \mathcal{B}^2_h}  \mathbb{E} [ \text{dist}^{2q}(x_k, \mathcal{X})] +  \alpha_k  (2 - \alpha_k \mathcal{L}) \mathbb{E}\left[ F(x_k) - F(\bar{x}_{k}) \right]\\
		& \leq \mathbb{E} [ \|v_{k-1} -\bar{v}_{k-1}\|^2 ]  + \alpha_k^2 \mathcal{B}^2.
	\end{align*}
	\red{Together with the fact that $\alpha_k(2-\alpha_k \mathcal{L}) <1$, it yields:
	\begin{align*}
		& \mathbb{E}\left[\|v_{k} - \bar{v}_{k} \|^2\right] + C_{\beta,c,\mathcal{B}_h}\alpha_k(2-\alpha_k \mathcal{L}) \mathbb{E} [ \text{dist}^{2q}(x_k, \mathcal{X})] +  \alpha_k  (2 - \alpha_k \mathcal{L}) \mathbb{E}\left[ F(x_k) - F(\bar{x}_{k}) \right]\\
		& \leq \mathbb{E} [ \|v_{k-1} -\bar{v}_{k-1}\|^2 ]  + \alpha_k^2 \mathcal{B}^2.
	\end{align*}
}
	Summing  this relation from $t=1:k$, we get:
	\begin{align*}
		& \mathbb{E}\left[\|v_{k} - \bar{v}_{k} \|^2 \right]  + C_{\beta,c,\mathcal{B}_h} \sum_{t=1}^{k} \red{\alpha_t  (2 - \alpha_t \mathcal{L})} \mathbb{E}\left[ \text{dist}^{2q}(x_{t}, \mathcal{X}) \right] \\
		& \quad+ \sum_{t=1}^{k} \red{\alpha_t  (2 - \alpha_t \mathcal{L})}   \mathbb{E}\left[  F(x_t) - F^*  \right]  \leq   \| v_0 - \bar{v}_{0} \|^2  +  \mathcal{B}^2 \sum_{t=1}^{k} \alpha_t^2.
	\end{align*}
From the definition of the  average sequence $\hat{x}_k$ and the  convexity of $F$ and of $ \text{dist}^2(\cdot, \mathcal{X})$, we get sublinear rate in expectation for the average sequence in terms of  optimality: 
	\begin{align*}
		&  \mathbb{E}\left[ F(\hat{x}_k ) - F^*  \right] \leq   \sum_{t=1}^{k} \frac{\red{\alpha_t  (2 - \alpha_t \mathcal{L})}}{S_k}   \mathbb{E}\left[  F(x_t) - F^*  \right]    \leq   \frac{\| v_0 - \bar{v}_{0} \|^2}{S_k}  +  \mathcal{B}^2 \frac{\sum_{t=1}^{k} \alpha_t^2}{S_k}.
	\end{align*}  
	Also by using Jensen's inequality and $q \geq 1$, we have:
	\begin{align*}
		& C_{\beta,c,\mathcal{B}_h} \left( \mathbb{E}\left[ \text{dist}^{2}(\hat{x}_k, \mathcal{X}) \right] \right)^q \leq  C_{\beta,c,\mathcal{B}_h} \mathbb{E}\left[ \text{dist}^{2q}(\hat{x}_k, \mathcal{X}) \right] \\
		& \leq C_{\beta,c,\mathcal{B}_h} \sum_{t=1}^{k} \frac{\red{\alpha_t  (2 - \alpha_t \mathcal{L})}}{S_k} \mathbb{E}\left[ \text{dist}^{2q}(x_{t}, \mathcal{X}) \right]  \leq   \frac{\| v_0 - \bar{v}_{0} \|^2}{S_k}  +  \mathcal{B}^2 \frac{\sum_{t=1}^{k} \alpha_t^2}{S_k}.
	\end{align*}
These conclude  our statements.                                             
\end{proof}
%%%SLAVA IS HERE
\noindent  For stepsize $\alpha_k=\frac{\alpha_0}{(k+1)^\gamma}$, with  $\gamma \in [1/2, 1)$  \red{and $\alpha_0$ satisfying  \eqref{eq:alk}}, we have:
\[ \frac{1}{\alpha_0}S_k \red{\overset{\eqref{eq:alk}}{\geq}} \frac{1}{\alpha_0} \sum_{t=1}^{k} \alpha_t \geq {\cal O}(k^{1-\gamma}) \quad \text{and} \quad \frac{1}{\alpha_0^2}\sum_{t=1}^{k} \alpha_t^2 \leq 
\begin{cases}
	{\cal O}(1) \;\;  \text{if} \; \gamma>1/2 \\
	{\cal O}(\ln(k)) \;\;\;  \text{if} \; \gamma=1/2. 
\end{cases} 
\] 
\noindent Consequently, for  $\gamma \in (1/2, 1)$ we obtain from Theorem \ref{th:nonstrongconv} the following sublinear convergence rates: 
\begin{align}\label{optimal_tau}
	& \mathbb{E}\left[ ( F(\hat{x}_k ) - F^* ) \right] \leq    \frac{\| v_0 - \bar{v}_{0} \|^2}{\alpha_0 \red{{\cal O}(k^{1-\gamma})}}  +  \frac{\alpha_0 \mathcal{B}^2\red{{\cal O}(1)}}{\red{{\cal O}(k^{1-\gamma})}}, \\
	& \mathbb{E}\left[ \text{dist}^2(\hat{x}_k, \mathcal{X}) \right]  \leq  \left(\frac{1}{C_{\beta,c,\mathcal{B}_h}\cdot \red{ \alpha_0 {\cal O}(k^{1-\gamma})}}\right)^{\frac{1}{q}} \left[ \|v_0 - \bar{v}_0\|^{\frac{2}{q}} + (\alpha_0^2 \mathcal{B}^2 \red{{\cal O}(1)})^{\frac{1}{q}}  \right].\nonumber
\end{align}

\noindent For the particular choice $\gamma=1/2$ we can perform the same analysis as before and obtain similar convergence bounds (\red{by replacing ${\cal O}(1)$ with $	{\cal O}(\ln(k))$}). Now, if we neglect the logarithmic terms, we get exactly the same rates as in \eqref{optimal_tau}, but replacing $k^{1-\gamma}$ with $k^{1/2}$. Hence, we omit the details for this case.  \\   
%\begin{align*}
%	& \mathbb{E}\left[ ( F(\hat{x}_k ) - F^* ) \right] \leq  \left\{\begin{array}{ll} {\cal O} \left( \sqrt{\frac{N}{\tau_1}} \frac{1}{ k^{1/2}} \right) & \text{if}\;\; \alpha_0 = \frac{\|v_0 - \bar{v}_0\|}{\mathcal{B}}\\
%	{\cal O} \left( \frac{N}{\tau_1} \frac{1}{ k^{1/2}} \right) & \text{if}\;\; \alpha_0 =\frac{1-\delta}{\mathcal{L}}, 	
%	\end{array}\right.\\
%    &\mathbb{E}\left[ \text{dist}^2(\hat{x}_k, \mathcal{X}) \right]  \leq {\cal O} \left(\frac{m}{\tau_2 k^{1/2}} \right)^{\frac{1}{q}}.
%\end{align*}

\noindent Minimizing the right hand side of the bound for optimality in \eqref{optimal_tau} w.r.t. $\alpha_0$, we get an optimal choice for the initial stepsize, i.e., $\alpha_0^* = \frac{\| v_0 - \bar{v}_{0} \|}{\mathcal{B}}$. Since $\alpha_0$ must be in \red{$\left(0, \min \left(\frac{1}{2}, \frac{1 - \sqrt{(1 - \mathcal{L})_+}}{\mathcal{L}}\right)\right)$}, then we consider $\alpha^*_0 = \min \left( \frac{\| v_0 - \bar{v}_{0} \|}{\mathcal{B}}, \red{\min \left(\frac{1}{2}, \frac{1 - \sqrt{(1 - \mathcal{L})_+}}{\mathcal{L}}\right)} - \delta \right)$ for some $\delta \in (0,\frac{1}{2})$. We distinguish two cases:

\noindent \textbf{Case 1:} If $\alpha_0^* = \frac{\mathcal{R}_0}{\mathcal{B}} \le \red{\min \left(\frac{1}{2}, \frac{1 - \sqrt{(1 - \mathcal{L})_+}}{\mathcal{L}}\right)} - \delta$, where $\mathcal{R}_0$ is an estimate of  $\| v_0 - \bar{v}_{0} \|$, then the expressions for the rates from \eqref{optimal_tau} are \red{(after ignoring ${\cal O}(1)/ {\cal O}(\ln (k))$ terms)}:
\begin{align*}
	& \mathbb{E}\left[ ( F(\hat{x}_k ) - F^* ) \right] \leq  \frac{\mathcal{B} \| v_0 - \bar{v}_{0} \|^2}{\mathcal{R}_0 \red{{\cal O}(k^{1-\gamma})}}  +  \frac{\mathcal{R}_0 \mathcal{B}}{ \red{{\cal O}(k^{1-\gamma})}},\\
	& \mathbb{E}\left[ \text{dist}^2(\hat{x}_k, \mathcal{X}) \right]  \leq  \red{\left(\frac{\mathcal{B}}{C_{\beta,c,\mathcal{B}_h} \mathcal{R}_0 \cdot \mathcal{O}(k^{1-\gamma})}\right)^{\frac{1}{q}} \left[ \|v_0 - \bar{v}_0\|^{\frac{2}{q}} + (\mathcal{R}_0 )^{\frac{2}{q}}  \right].}
\end{align*}   
Using the definition of $C_{\beta,c,\mathcal{B}_h}$ and  replacing the values for $\mathcal{L}$, $\mathcal{B}$, $\mathcal{B}_h$ and $c$ from Theorem \ref{sampling_cases} for both types of samplings, i.e., partition or $\tau_1$-, $\tau_2$-nice samplings,  we get:
\begin{align*}
	& \mathbb{E}\left[ ( F(\hat{x}_k ) - F^* ) \right] \leq  \sqrt{\frac{N}{\tau_1}}\frac{B}{ \mathcal{O}(k^{1-\gamma})} \left( \frac{\| v_0 - \bar{v}_{0} \|^2}{\mathcal{R}_0} + \mathcal{R}_0 \right), \\
	& \mathbb{E}\left[ \text{dist}^2(\hat{x}_k, \mathcal{X}) \right]  \leq \red{ \left(\sqrt{\frac{N}{\tau_1}} \frac{ B m \bar{c} \max_{ j=1:m}^2 B_j}{\tau_2 \cdot  \beta(2-\beta) \mathcal{R}_0 \cdot\mathcal{O}(k^{1-\gamma})}\right)^{\frac{1}{q}} \left[ \|v_0 - \bar{v}_0\|^{\frac{2}{q}} + (\mathcal{R}_0)^{\frac{2}{q}} \right].}
\end{align*}

\noindent \textbf{Case 2:} If $\alpha_0^* = \red{\min \left(\frac{1}{2}, \frac{1 - \sqrt{(1 - \mathcal{L})_+}}{\mathcal{L}}\right)} - \delta < \frac{\| v_0 - \bar{v}_{0} \|}{\mathcal{B}}$, for some $\delta \in (0,1/2)$. Then, the expressions for the rates from \eqref{optimal_tau} are \red{(after ignoring ${\cal O}(1)/ {\cal O}(\ln (k))$ terms)}:
\begin{align}
	& \mathbb{E}\left[ ( F(\hat{x}_k ) - F^* ) \right] \leq    \frac{\| v_0 - \bar{v}_{0} \|^2}{\alpha_0^*\cdot {\cal O} (k^{1-\gamma})}  +  \frac{\alpha_0^*\cdot \mathcal{B}^2}{  {\cal O} (k^{1-\gamma})} \le \frac{2 \| v_0 - \bar{v}_{0} \|^2}{\alpha_0^*\cdot {\cal O} (k^{1-\gamma})} , \label{eq:11}\\
	& \mathbb{E}\left[ \text{dist}^2(\hat{x}_k, \mathcal{X}) \right]  \leq  \left(\frac{1}{C_{\beta,c,\mathcal{B}_h} \red{\alpha_0^*}\cdot  {\cal O} (k^{1-\gamma}) }\right)^{\frac{1}{q}} \left[ \|v_0 - \bar{v}_0\|^{\frac{2}{q}} + \left((\alpha_0^*)^2\mathcal{B}^2 \right)^{\frac{1}{q}} \right] \nonumber\\
	& \qquad \qquad \qquad\quad \le \left(\frac{1}{C_{\beta,c,\mathcal{B}_h}\red{\alpha_0^*} \cdot  {\cal O} (k^{1-\gamma})}\right)^{\frac{1}{q}} \left[ 2\|v_0 - \bar{v}_0\|^{\frac{2}{q}}\right]. \label{eq:12}
\end{align}

\noindent \red{Consider the case when  $\alpha^*_0 = \frac{1}{2} - \delta$, from \eqref{eq:11}, and \eqref{eq:12}, we have:
\begin{align*}
	& \mathbb{E}\left[ ( F(\hat{x}_k ) - F^* ) \right] \leq \frac{4 \| v_0 - \bar{v}_{0} \|^2}{ (1 - 2 \delta) {\cal O} (k^{1-\gamma})} , \\
	& \mathbb{E}\left[ \text{dist}^2(\hat{x}_k, \mathcal{X}) \right] \le \left(\frac{2}{C_{\beta,c,\mathcal{B}_h} (1 - 2 \delta) {\cal O} (k^{1-\gamma})}\right)^{\frac{1}{q}} \left[ 2\|v_0 - \bar{v}_0\|^{\frac{2}{q}}\right].
\end{align*}
 Using the definition of $C_{\beta,c,\mathcal{B}_h}$ and the expressions for $\mathcal{B}_h$ and $c$ from Theorem \ref{sampling_cases} for the partition  or  $\tau_1$-, $\tau_2$-nice samplings, we get:
\begin{align*}
	& \mathbb{E}\left[ ( F(\hat{x}_k ) - F^* ) \right] \leq  \frac{4 \| v_0 - \bar{v}_{0} \|^2}{(1 - 2 \delta) {\cal O} (k^{1-\gamma})}, \\
	& \mathbb{E}\left[ \text{dist}^2(\hat{x}_k, \mathcal{X}) \right]  \leq  \left(\frac{ 2m \bar{c} \max_{ j=1:m}^2 B_j}{\tau_2 \cdot  \beta(2-\beta)\cdot (1 - 2 \delta) {\cal O} (k^{1-\gamma})}\right)^{\frac{1}{q}} \left[ 2 \|v_0 - \bar{v}_0\|^{\frac{2}{q}} \right].
\end{align*}
}

\noindent\red{When  $\alpha^*_0 = \frac{1 - \sqrt{(1 - \mathcal{L})_+}}{\mathcal{L}} -\delta$, from \eqref{eq:11}, and \eqref{eq:12}, we have:
	\begin{align*}
		& \mathbb{E}\left[ ( F(\hat{x}_k ) - F^* ) \right] \leq \frac{2\mathcal{L} \| v_0 - \bar{v}_{0} \|^2}{(1 - \sqrt{(1 - \mathcal{L})_+} - \delta \mathcal{L}) {\cal O} (k^{1-\gamma})} , \\
		& \mathbb{E}\left[ \text{dist}^2(\hat{x}_k, \mathcal{X}) \right] \le \left(\frac{2\mathcal{L}}{C_{\beta,c,\mathcal{B}_h} (1 - \sqrt{(1 - \mathcal{L})_+} - \delta \mathcal{L}) {\cal O} (k^{1-\gamma})}\right)^{\frac{1}{q}} \left[ 2\|v_0 - \bar{v}_0\|^{\frac{2}{q}}\right].
	\end{align*}
}

\noindent Using the definition of $C_{\beta,c,\mathcal{B}_h}$ and the expressions for $\mathcal{L}$, $\mathcal{B}$, $\mathcal{B}_h$ and $c$ from Theorem \ref{sampling_cases} for the partition  or  $\tau_1$-, $\tau_2$-nice samplings, we get:
\begin{align*}
	& \mathbb{E}\left[ ( F(\hat{x}_k ) - F^* ) \right] \leq  \frac{N}{\tau_1}  \frac{2 L\| v_0 - \bar{v}_{0} \|^2}{\red{\left(1- \sqrt{(1-\frac{N}{\tau_1}L)_+} - \delta \frac{N}{\tau_1}L \right)}\cdot{\cal O} (k^{1-\gamma})}, \\
	& \mathbb{E}\left[ \text{dist}^2(\hat{x}_k, \mathcal{X}) \right] \\
	& \leq \! \left( \!\red{\frac{N}{\tau_1}} \frac{m}{\tau_2}\frac{ L \bar{c} \max_{ j=1:m}^2 B_j}{  \beta(2-\beta) \red{\left(1- \sqrt{(1-\frac{N}{\tau_1}L)_+} - \delta \frac{N}{\tau_1}L \right)} \cdot{\cal O} (k^{1-\gamma})} \!\right)^{\frac{1}{q}} \!\left[ 2 \|v_0 - \bar{v}_0\|^{\frac{2}{q}} \right].
\end{align*}

\noindent Note that for the initial stepsize choices  $\alpha_0^* = \frac{\red{\mathcal{R}_0 }}{\mathcal{B}}\; \text{or}\; \alpha_0^* = \red{ \frac{1 - \sqrt{(1 - \mathcal{L})_+}}{\mathcal{L}}} - \delta$ and  for the two particular choices of the sampling (partition  or nice samplings), we obtain convergence rates depending explicitly on mini-batch sizes $\tau_1$ and $\tau_2$, namely $\left(\sqrt{\frac{N}{\tau_1}}, \left( \red{\sqrt{\frac{N}{\tau_1}}} \frac{m}{\tau_2} \right)^{1/q}\right)$ or  $ \left(\frac{N}{\tau_1}, \left(\red{\frac{N}{\tau_1}} \frac{m}{\tau_2} \right)^{1/q}\right)$,  respectively.  Hence, in these settings we have linear dependence on the mini-batch sizes $(\tau_1, \tau_2)$ for algorithm \texttt{Mini-batch SSP}.\\  

\noindent Furthermore, since in the convex case we can consider a  stepsize sequence $\alpha_k=\frac{\alpha_0}{(k+1)^\gamma}$, then for $\alpha_0 = \frac{\red{\mathcal{R}_0 }}{\mathcal{B}}\; \text{or}\; \alpha_0 = \red{\frac{1 - \sqrt{(1 - \mathcal{L})_+}}{\mathcal{L}}} - \delta$ one can notice immediately that our stepsize sequence $\alpha_k$ also depends linearly on the mini-batch size $\tau_1$ for the two particular choices of sampling (partition  or nice samplings), i.e.,  $\alpha_k=\mathcal{O} \left( \frac{\tau_1}{N(k+1)^\gamma} \right)$. \\

\noindent Finally, one can notice that when $B=0$, from  Theorem \ref{th:nonstrongconv} improved rates can be derived for \texttt{Mini-batch SSP} in the convex case. For example, for stepsize $\alpha_k=\frac{\alpha_0}{(k+1)^\gamma}$, with  $\gamma \in [0, 1)$ and $\alpha_0 = \red{\min \left(\frac{1}{2}, \frac{1 - \sqrt{(1 - \mathcal{L})_+}}{\mathcal{L}}\right)} - \delta$, we obtain convergence rates for  $\hat{x}_k$  in optimality and feasibility violation of order  ${\cal O} \left( \frac{N}{\tau_1 k^{1-\gamma}} \right)$ and ${\cal O} \left( \frac{\red{N}m}{\red{\tau_1} \tau_2 k^{1-\gamma}} \right)^{\frac{1}{q}}$, respectively. In particular,  for $\gamma=0$ these rates become of order   ${\cal O} \left( \frac{N}{\tau_1 k} \right)$ and ${\cal O} \left( \frac{\red{N}m}{\red{\tau_1}\tau_2 k} \right)^{\frac{1}{q}}$. \\

\noindent In conclusion, by specializing our Theorem \ref{th:nonstrongconv} to different mini-batching strategies, such as partition  or nice samplings, we derive explicit expressions for the stepsize $\alpha_k$ as a function of the mini-batch size and, consequently,  convergence rates depending linearly  on the mini-batch sizes $(\tau_1, \tau_2)$. Hence, Theorem \ref{th:nonstrongconv} shows that  a mini-batch variant of the stochastic subgradient projection scheme is more beneficial than the nonmini-batch variant. 

%%%%%%%%%%%%%%%%%%%%%%%%%%%%%%%%%%%%%%%%%%%%%%%%%%%%%

\subsection{Convergence analysis: \red{strongly} convex objective function} \label{sec3.2}
In this section, we additionally assume the inequality from Assumption \ref{assumption2} holds.  The next lemma derives an improved recurrence  for the sequence $v_k$  under  the strongly convex assumption. The proof is similar to Lemma 8 in \cite{NecSin:21}. 
\begin{lemma}
	\label{lem:reccStrongCon}
	Let  $f_i, \; g_i$, with $i=1:N$ and  $h_j$, with $j=1:m$,  be  convex functions.  Additionally,  Assumptions \ref{assumption1}--\ref{assumption4} hold, \red{with $\mu>0$},  and the random vectors $\zeta,\; \xi$ are nonnegative. Define \blue{$k_0 = \lfloor\frac{8 \mathcal{L}}{\mu} - 1\rfloor$}, $\beta \in \left(0,2\right)$, $\theta_{\mathcal{L},\mu} \!=\! 1  \!-\! \mu/(4\mathcal{L})$  and $\alpha_k \!=\! \frac{4}{\mu} \gamma_{k}$,  where $\gamma_k$ is given~by:  
	\begin{equation*}
		\gamma_{k} = \left\{\begin{array}{ll}\frac{\mu}{4\mathcal{L}} & \text{\emph{if}}\;\; k\leq k_0\\
			\frac{2}{k+1}& \text{\emph{if}}\;\; k > k_0.
		\end{array}
		\right.
	\end{equation*}
	Then, the iterates of Algorithm Mini-batch SSP satisfy the following recurrence:
	\begin{align*}
		& \mathbb{E}[\|v_{k_0} - x^*\|^2] 
		\leq    \left\{\begin{array}{ll}    \frac{\mathcal{B}^2}{\mathcal{L}^2}    & \text{\emph{if}} \;\; \theta_{\mathcal{L},\mu}  \leq 0\\
			\theta_{\mathcal{L},\mu}^{k_0}  \|v_{0} - x^*\|^2 +  \frac{1 - \theta_{\mathcal{L},\mu}^{k_0}  }{1 - \theta_{\mathcal{L},\mu}}   \red{\left( 1 + \frac{2}{C_{\beta,c,\mathcal{B}_h} \theta_{\mathcal{L},\mu}} \right)} \frac{\mathcal{B}^2}{\mathcal{L}^2}  & \text{\emph{if}} \;\; \theta_{\mathcal{L},\mu}  >  0, 
		\end{array}
		\right. \\
		&\mathbb{E}[\|v_k - x^*\|^2] +\gamma_{k} \mathbb{E}[\| x_k - x^*\|^2] + \red{\frac{1}{6}} C_{\beta,c,\mathcal{B}_h}\mathbb{E}[\emph{dist}^{2q}(x_{k}, \mathcal{X})]\\
		&  \leq \left(1 - \gamma_{k}\right)\mathbb{E}[\|v_{k-1}- x^*\|^2] + \red{\left( 1 + \frac{6}{C_{\beta,c,\mathcal{B}_h} } \right)} \frac{16}{\mu^2}\gamma_{k}^2 \mathcal{B}^2  \quad  \forall k> k_0.
	\end{align*}
\end{lemma}

\noindent Let us define for $k \geq k_0+1$ the sum: 
\begin{align*}
	S_k = \sum_{t=k_0+1}^{k} (t+1)^2 \sim \mathcal{O} (k^3 + k_0^2k + k^2k_0)
\end{align*}
and  the corresponding average sequences: 
\begin{align*}
	&\hat{x}_k = \frac{\sum_{t=k_0+1}^{k} (t+1)^2 x_t}{S_k}, \quad	
%	\hat{x}_k^* = {S_k}^{-1} \sum_{t=k_0+1}^{k} (t+1)^2 \bar{x}_t \in \mathcal{X}^* \\
	\text{and} \quad   \hat{w}_k = \frac{\sum_{t=k_0+1}^{k} (t+1)^2 \Pi_{\mathcal X}(x_t)}{S_k} \in \mathcal{X}. 
\end{align*}

\begin{theorem}
	\label{th:strcase2}
	Let  $f_i, \; g_i$, with $i=1:N$ and  $h_j$, with $j=1:m$, be  convex functions.  Additionally,  Assumptions \ref{assumption1}--\ref{assumption4} hold and the random vectors $\zeta,\; \xi$ are non-negative.  Further, consider the stepsizes-switching rule  $\alpha_k = \min\left(\frac{1}{\mathcal{L}}, \frac{8}{\mu (k + 1)}\right)$, $\beta \in \left(0,2\right)$ and \blue{$k_0 = \lfloor\frac{8 \mathcal{L}}{\mu} - 1\rfloor$}. Then, for $k> k_0$ we have the following sublinear convergence rates  for the  average sequence $	\hat{x}_k $ in terms of  optimality and feasibility violation for problem \eqref{eq:prob} (keeping only the dominant terms):
	
	\begin{align*}
		&\mathbb{E}\left[\|\hat{x}_k - x^*\|^2 \right] \leq \red{\mathcal{O}\left(      \frac{\mathcal{B}^2}{\mu^2  C_{\beta,c,\mathcal{B}_h}\, (k + 1)} \right)},\\
		&\mathbb{E}\left[\emph{dist}^2(\hat{x}_k, \mathcal{X})\right]  \leq \mathcal{O}\left(  \frac{\mathcal{B}^{2/q}}{\mu^{2/q} \red{C_{\beta,c,\mathcal{B}_h}^{2/q}} (k+1)^{2/q}}\right).
	\end{align*}
\end{theorem}

\begin{proof}
	Using Lemma \ref{lem:reccStrongCon}, we get the recurrence:
	\begin{align*}
		(k+1)^2\mathbb{E}[\|v_k - x^*\|^2] & +2(k+1) \mathbb{E}[\| x_k - x^*\|^2]  + \frac{C_{\beta,c,\mathcal{B}_h}}{6} (k+1)^2\mathbb{E}[\text{dist}^{2q}(x_{k}, \mathcal{X})]\\
		&  \leq k^2 \mathbb{E}[\|v_{k-1}- x^*\|^2] + \red{\left( 1 + \frac{6}{C_{\beta,c,\mathcal{B}_h} } \right)} \frac{64}{\mu^2} \mathcal{B}^2 \quad  \forall k> k_0.
	\end{align*}
	Summing  this inequality from $k_0+1$ to $k$ and using linearity of the expectation operator and convexity of the norm, we get: 
	\begin{align*}
		&{(k+1)^2}\mathbb{E}[\|v_k - x^*\|^2]+ \frac{2 S_k}{(k+1)}\mathbb{E} [\|\hat{x}_k- x^*\|^2] + \frac{S_k C_{\beta,c,\mathcal{B}_h}}{6}  \mathbb{E}[\|\hat{w}_k-\hat{x}_k\|^{2q}]  \nonumber\\
		& \leq (k_0+1)^2 \mathbb{E}[\|v_{k_0} -x^* \|^2] + \red{\left( 1 + \frac{6}{C_{\beta,c,\mathcal{B}_h} } \right)} \frac{64}{\mu^2} \mathcal{B}^2(k-k_0). 
	\end{align*}
	After simple calculations and keeping only the dominant terms, we get the following convergence rate for the average sequence $\hat{x}_k$ in terms of optimality:
	\begin{align*}
		&\mathbb{E} [\|\hat{x}_k - x^*\|^2] \leq \mathcal{O} \left( \frac{\mathcal{B}^2}{\mu^2 \red{C_{\beta,c,\mathcal{B}_h} }\, (k+1)} \right), \\
		& \left( \mathbb{E}[\|\hat{w}_k-\hat{x}_k\|^{2}] \right)^q \leq  \mathbb{E}[\|\hat{w}_k-\hat{x}_k\|^{2q}] \leq \mathcal{O}\left( \frac{\mathcal{B}^2}{\mu^2\red{C^2_{\beta,c,\mathcal{B}_h}} \,  (k+1)^2 }\right).
	\end{align*}
	Since $\hat{w}_k \in \mathcal{X}$,  we get the following convergence rate for the average sequence $\hat{x}_k$ in terms of  feasibility violation: 
	\begin{align*}
		\mathbb{E}[\text{dist}^2(\hat{x}_k,\mathcal{X})] &\leq \mathbb{E}[\|\hat{w}_k-\hat{x}_k\|^2]  \leq \mathcal{O}\left( \frac{\mathcal{B}^2}{\mu^2\red{C^2_{\beta,c,\mathcal{B}_h}} \, (k+1)^2 }\right)^{\frac{1}{q}}.
%		& \le \mathcal{O} \left( \frac{\mathcal{B}^{2/q}}{\mu^{2/q} \red{C_{\beta,c,\mathcal{B}_h}^{2/q}} \, k^{2/q}} \right). 
	\end{align*}
%	Furthermore, since $	\hat{x}_k^* \in \mathcal{X}^*$ and  using the inequality $\|a+b\|^2 \leq 2 \|a\|^2 +2\|b\|^2$, we also get convergence rate for the average sequence $\hat{x}_k$ in terms of optimality:
%	\begin{align*}
%		\mathbb{E}[\text{dist}^2(\hat{x}_k,\mathcal{X}^*)] & \leq \mathbb{E} [\| \hat{x}_k - 	\hat{x}_k^* \|^2] \leq 2\mathbb{E} [\|\hat{w}_k-\hat{x}_k^* \|^2] + 2\mathbb{E}[\|\hat{w}_k-\hat{x}_k\|^2]\\
%		&  \leq \mathcal{O}\left( \frac{\mathcal{B}^{2/q}}{\mu^{2/q} C_{\beta,c,\mathcal{B}_h}^{1/q} (k^2 + k k_0 + k_0^2)^{1/q}}  +   \frac{\mathcal{B}^2}{\mu^2(k-k_0)}    \right).
%	\end{align*}
These prove our statements. 	
\end{proof}

\noindent Note that our previous theoretical convergence analysis naturally imposes a stepsize-switching rule which describes when one should switch from a constant regime (depending on mini-batch size $\tau_1$) to a decreasing stepsize regime, i.e.,  $\alpha_k = \min\left(\frac{1}{\mathcal{L}}, \frac{8}{\mu (k + 1)}\right)$. For the particular choice of the stepsize  $\beta =1$, we have (see \eqref{C}):
\begin{align*}
	C_{1,c,\mathcal{B}_h}= \left(\frac{1}{c\mathcal{B}_h^2}\right) > 0,
\end{align*}
since  we can always choose $c$ such that $c\mathcal{B}_h^2 >1$.  Using this expression  in the convergence rates of Theorem \ref{th:strcase2}, we obtain: 
\begin{align*}
    &\mathbb{E}\left[\|\hat{x}_k - x^*\|^2 \right] \leq \red{\mathcal{O}\left(  \frac{\mathcal{B}^2\red{(c\mathcal{B}_h^2)}}{\mu^2 (k + 1)}\right)},\\
	&\mathbb{E}\left[\text{dist}^2(\hat{x}_k, \mathcal{X})\right]  \leq \mathcal{O}\left( \frac{\mathcal{B}^2 \red{(c\mathcal{B}_h^2)^2}}{\mu^2 (k+1)^2} \right)^{1/q}.
\end{align*}

\noindent By replacing the values for $\mathcal{L}$, $\mathcal{B}$, $\mathcal{B}_h$ and $c$ from Theorem \ref{sampling_cases} for both types of sampling, i.e., partition or $\tau_1$, $\tau_2$-nice samplings, we get:
 \begin{align*}
&\mathbb{E}\left[\|\hat{x}_k - x^*\|^2 \right] \leq \red{\mathcal{O}\left( \frac{m}{\tau_2} \frac{N}{\tau_1} \cdot \frac{B^2 \bar{c} \max_{ j=1:m}^2 B_j}{\mu^2 (k+1)}\right)},\\
&\mathbb{E}\left[\text{dist}^2(\hat{x}_k, \mathcal{X})\right]  \leq \red{ \mathcal{O}\left(\left(\frac{m}{\tau_2}\right)^2 \frac{N}{\tau_1} \frac{B^2 \bar{c}^2 \max_{ j=1:m}^4 B_j }{ \mu^2 (k+1)^2 } \right)^{1/q}.}
\end{align*}

\noindent One can easily see that also in this case the obtained rates have linear dependence on the mini-batch sizes $(\tau_1, \tau_2)$. Therefore,  Theorem \ref{th:strcase2} also proves that in the quadratic growth convex  case a mini-batch variant of the  stochastic subgradient projection scheme with  a stepsize-switching rule brings  benefits over the nonmini-batch~variant.

\section{Numerical simulations}
\noindent In this section, we consider a general quadratic program with quadratic constraints:  

\begin{equation}
\label{eq:Numprob}
\begin{array}{rl}
\min_{x \in \mathbb{R}^n} & \frac{1}{2} \|Ax - b\|^2 + \|\Delta x\|_1\\
\text{subject to } & Cx + d \ge 0, \;\; c_i^{T}x +d_i \ge \|Q_i^{-1/2} x\| \quad \forall i = 1:m,
\end{array}
\end{equation}

\noindent with the matrices $A \in \mathbb{R}^{N \times n}$, $\Delta  \in \mathbb{R}^{N \times n}$, $C \in \mathbb{R}^{m \times n}$,  $Q_i \in \mathbb{R}^{m_i \times n}$ and $c_i \in \mathbb{R}^n$, with $i=1:m$. One can notice that this problem fits into our general modeling framework \eqref{eq:prob}  (e.g., define $f_i(x) = 1/2(a_i^Tx-b_i)^2$, with $a_i$ the $i$th row of matrix $A$, $g_i(x) =  \|\delta_i^T x\|_1$, with $\delta_i$ the $i$th row of matrix $\Delta$, for all $i=1:N$, and $h_j$ are either linear or quadratic constraints, for all $j=1:2m$). Moreover, \eqref{eq:Numprob} is a general constrained Lasso problem which  appears in many applications from machine learning, signal processing and statistics, see \cite{BhaGra:04,NecSin:21, JamRus:19, GaiZho:18, HuLin:15}. In particular  if one considers appropriate matrices $A$, $\Delta$, $C$ and $Q_i$, one can recast the robust (sparse) SVM problem from \cite{BhaGra:04,NecSin:21} as  problem \eqref{eq:Numprob}. Indeed, the robust (sparse) SVM problem is defined as \cite{BhaGra:04,NecSin:21}:
\begin{align*}
& \min_{w,d,u}  \; \frac{\lambda}{2} \|w\|^2 + \delta \sum_{i=1}^m u_i +  \|w\|_1  \\
& \text{subject to:}  \; u\ge 0, \;  y_i (w^T \bar{z}_i + d) \geq 1 \!-\! u_i, \\
& \qquad \qquad \;\;\;    y_i(w^T \bar{z}_i + d) \geq \|Q_i^{-1/2} w\| -  u_i \;\;\; \forall i=1\!:\!m,
\end{align*}
where $(\bar z_i)_{i=1}^m$ is the training dataset, $(y_i)_{i=1}^m \in \{-1,1\}$ are the corresponding labels, $Q_i$'s are diagonal matrices with positive entries, \blue{$u_i$'s are the slack variables}, $\delta>0$ and $(w,d) \in \mathbb{R}^{n} \times \mathbb{R}$ are the parameters of the hyperplane to separate the data \blue{ (see Section $5.1$ in \cite{NecSin:21} for more details)}.\\

\noindent In the numerical experiments we consider random matrices  $A$ and $C$ and diagonal matrices $\Delta$ and $Q_i$, all generated from normal distributions. We consider as epoch  $ \max \left(  \frac{N}{\tau_1}, \frac{m}{\tau_2}\right)$ iterations of \texttt{Mini-batch SSP} algorithm and our stopping criteria are $\|\max(0, h(x))\|_2 \le 10^{-2}$ and $F(x) - F^* \leq 10^{-2}$ (we consider \texttt{CVX} solution \cite{GraBoy:13} for computing $F^*$, when \texttt{CVX} finishes in a reasonable time). The codes are written in Matlab and run on a PC with i7 CPU at 2.1 GHz and 16 GB RAM memory.\\ 

\noindent Figure \ref{fig_big_d} shows the convergence behaviour of \texttt{Mini-batch SSP} algorithm  along epochs with four different choices for mini-batch sizes $(\tau_1, \tau_2)$ as $(1,1),\; (20,80), \; (60,160)$ and $(N = 120,m = 240)$ in terms of optimality (left) and feasibility (right) for solving the constrained Lasso problem \eqref{eq:Numprob} with $N = 120, n=110,  m = 240$. As we can see from this figure, increasing the minibatch sizes ($\tau_1,\tau_2$) leads to better convergence than the nonmini-batch counterpart, as our theory also predicted. 
%\begin{figure}[ht]
%	\centering
%	\includegraphics[height=6cm, width=7cm]{mini-batchOpt2.eps}
%	\includegraphics[height=6cm, width=7cm]{mini-batchFeas2.eps}
%	%\includegraphics[height=5.5cm, width=6.5cm]{boxplot.eps}
%	\caption{Behaviour of \texttt{Mini-batch SSP} algorithm in terms of  optimality and feasibility of the problem when non-empty interior set is small (N = 120, m = 240).}
%	\label{fig_small_d}
%\end{figure}
\begin{figure}[ht]
	\centering 
	\includegraphics[height=6cm, width=7cm]{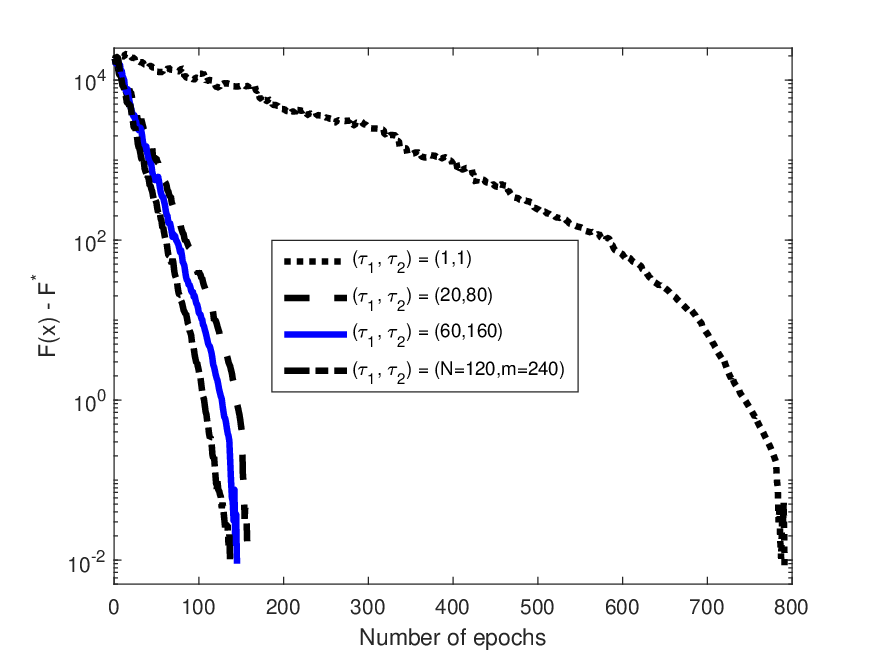}
	\includegraphics[height=6cm, width=7cm]{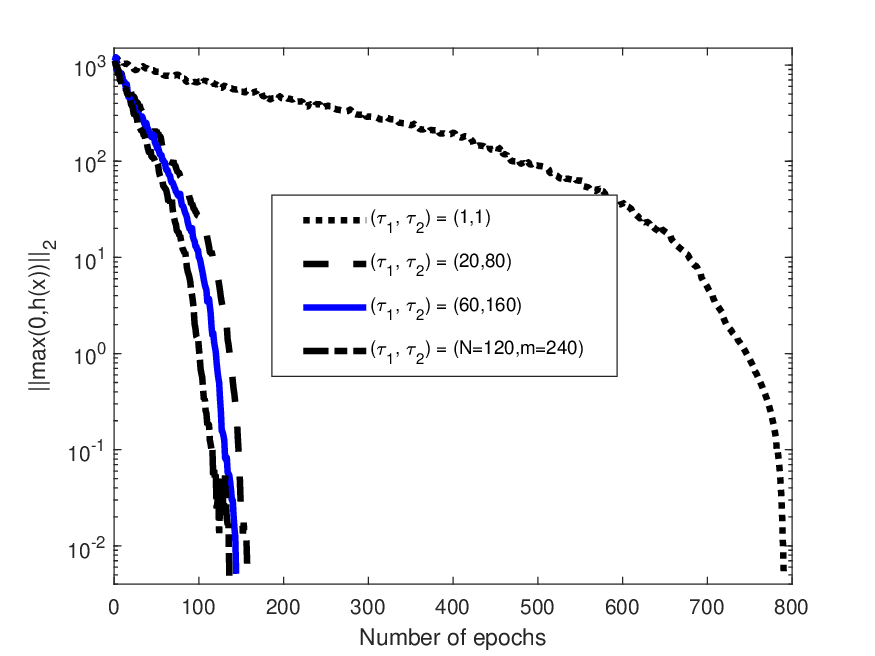}
	\caption{Behaviour of \texttt{Mini-batch SSP} algorithm in terms of optimality (left) and feasibility (right) for $N = 120, n=110,  m = 240$ and different mini-batch sizes ($\tau_1,\tau_2$).}
	\label{fig_big_d}
\end{figure}

\medskip	
\noindent  Finally, in Table \ref{tab_small_d}  we compare  \texttt{Mini-batch SSP} algorithm with \texttt{CVX} in terms of cpu time (in seconds) for solving  problem \eqref{eq:Numprob} over different dimensions of the problem ranging from several hundreds to thousands of functions ($N$) and constraints ($m$), respectively (note that if $N<n$, then the objective function $F$ is convex, otherwise $F$ is strongly convex). For  \texttt{Mini-batch SSP} algorithm we consider four different choices for mini-batch sizes and in the table we also give the number of epochs. The results we present in the table is the average of $10$ runs on the same problem.  From the table we observe that for some choices of mini-batch sizes \texttt{Mini-batch SSP} algorithm is even $10$ times faster than \texttt{CVX} ("*" means that \texttt{CVX} has not finished after 3 hours). Moreover,  \texttt{Mini-batch SSP} is much faster than its nonmini-batch counterpart.  

\begin{table}[ht]
	\centering
	\begin{tabular}{|c|c|l|l|l|}
		\hline
		\textbf{Sizes}                                                              & \multicolumn{1}{l|}{\textbf{\begin{tabular}[c]{@{}l@{}}\texttt{Mini-batch SSP}\\ sizes($\tau_1$, $\tau_2$)\end{tabular}}} & \textbf{epochs} & \textbf{cpu time} & \multicolumn{1}{l|}{\textbf{\begin{tabular}[c]{@{}l@{}}\texttt{CVX}\\ cpu time\end{tabular}}}     \\ \hline
		\multirow{4}{*}{\begin{tabular}[c]{@{}c@{}}N = 120,\\ m = 240,\\n = 110\end{tabular}} & (1, 1)                                                                                                                      &     655     &    0.17       & \multirow{4}{*}{} \\ \cline{2-4}
		& (20, 80)                                                                                                      &     148     &   0.06    &        1.44           \\ \cline{2-4}
		& (60, 160)                                                                                                   &       131       &  0.07         &              \\ \cline{2-4}
		& (N, m)                                                                                                        &     166      &  0.11       &                   \\ \hline
		\multirow{4}{*}{\begin{tabular}[c]{@{}c@{}}N = 100,\\ m = 240,\\ n = 110\end{tabular}} & (1, 1)                                                                                                                     &     1023      &        0.25       & \multirow{4}{*}{} \\ \cline{2-4}
		& (20, 80)                                                                                                         &     202     &       0.08      &     1.51             \\ \cline{2-4}
		& (60, 160)                                                                                                     &    175   &       0.08      &               \\ \cline{2-4}
		& (N, m)                                                                                                       &  357    &       0.21    &                   \\ \hline
		\multirow{4}{*}{\begin{tabular}[c]{@{}c@{}}N = 1200,\\ m = 2400,\\n = 1100\end{tabular}} & (1, 1)                                                                                                              &  8131 &      51.94      & \multirow{4}{*}{} \\ \cline{2-4}
		& (200, 800)                                                                                                 &    958      &    9.38   &   177.08       \\ \cline{2-4}
		& (600, 1600)                                                                                               &    713     &   9.59   &              \\ \cline{2-4}
		& (N, m)                                                                                                   & 2327      &    48.81    &                   \\ \hline
		\multirow{4}{*}{\begin{tabular}[c]{@{}c@{}}N = 1000,\\ m = 2400,\\ n = 1100\end{tabular}} & (1, 1)                                                                                                                &    13115   &  66.15       & \multirow{4}{*}{} \\ \cline{2-4}
		& (200, 800)                                                                                                  &   1983     &  14.70       &       179.67         \\ \cline{2-4}
		& (600, 1600)                                                                                                 &       1158  &      12.07  &            \\ \cline{2-4}
		& (N, m)                                                                                                      &   5771      &     61.33   &                   \\ \hline
		\multirow{4}{*}{\begin{tabular}[c]{@{}c@{}}N = 3600,\\ m = 7200,\\n = 3300\end{tabular}}                         & (1, 1)         & 19491   &     2008.60      & \multirow{4}{*}{} \\ \cline{2-4}
		& (600, 2400)                 &   298      &   52.94   &   *      \\ \cline{2-4}
		& (1800, 4800)           &    1432     &  387.91    &              \\ \cline{2-4}
		& (N, m)       &   1200   &    464.79   &                   \\ \hline
		\multirow{4}{*}{\begin{tabular}[c]{@{}c@{}}N = 3000,\\ m = 7200,\\ n = 3300\end{tabular}}          & (1, 1)             &   40168  &    3618.37    & \multirow{4}{*}{} \\ \cline{2-4}
		& (600, 2400)          &  2990    &  457.99     &        *     \\ \cline{2-4}
		& (1800, 4800)        &  2130     &  471.87     &            \\ \cline{2-4}
		& (N, m)              &  24903  & 7260.44      &                   \\ \hline
		%\multirow{3}{*}{\begin{tabular}[c]{@{}c@{}}N = 12000,\\ m = 24000,\\n = 11000\end{tabular}} & (1, 1)                                                                                                          &     41420   & 264852.55      & \multirow{4}{*}{} \\ \cline{2-4}
		%& (4000, 10000)                                                                                               &   28172   &    145342.11      &     *         \\ \cline{2-4}
		%& (N, m)                                                                                                    &       36129    &    240223.85   &                   \\ \hline
		%\multirow{3}{*}{\begin{tabular}[c]{@{}c@{}}N = 10000,\\ m = 24000,\\ n = 11000\end{tabular}} & (1, 1)                                                                                                        &     35099    &             & \multirow{4}{*}{} \\ \cline{2-4}
	%	& (4000, 10000)                                                                                           &       &              &       *            \\ \cline{2-4}
		%& (N, m)                                                                                                &         &             &                   \\ \hline
	\end{tabular}
	\caption{Comparison between \texttt{Mini-batch SSP} and \texttt{CVX} for different dimensions  and mini-batch sizes.}
	\label{tab_small_d}
\end{table}

\section{Conclusions}
In this paper we have considered a deterministic general finite sum composite optimization problem with many functional constraints. We have reformulated this problem into a  stochastic problem  for which the stochastic subgradient projection method  from \cite{NecSin:21} specializes  to an infinite array of mini-batch variants, each of which is associated with a specific probability law governing the data selection rule used to form mini-batches.  By specializing different mini-batching strategies,  we have derived exact expressions for the stepsizes as a function of  the mini-batch size and in some cases we have derived  stepsize-switching rules which describe when one should switch from a constant to a decreasing stepsize regime.  We  have also proved sublinear convergence rates for the mini-batch subgradient projection algorithm which depend explicitly on the mini-batch sizes and on the properties of the objective function.  Preliminary numerical results support the effectiveness of our method in practice. 

%\textcolor{red}{For future work, we want to also consider other sampling strategies as in \cite{GorRic:20} and derive explicit convergence results under different assumptions on the objective function.}

\section*{Funding}
The research leading to these results has received funding from: the NO Grants 2014–2021 RO-NO-2019-0184, under project ELO-Hyp, contract no. 24/2020; UEFISCDI PN-III-P4-PCE-2021-0720, under project L2O-MOC, nr. 70/2022 for N.K. Singh and I. Necoara.  The OP VVV project
CZ.02.1.01/0.0/0.0/16\_019/0000765 Research Center for Informatics for V. Kungurtsev.


\begin{thebibliography}{99}
	
\bibitem{AsiDuc:20}
	H. Asi, K. Chadha, G. Cheng and J. Duchi, \textit{Minibatch stochastic approximate proximal point methods}, Advances in Neural Information Processing Systems Conference,  2020.
	
	\bibitem[Bhattacharyya et al., 2004]{BhaGra:04}
	C.  Bhattacharyya,  L.R. Grate,  M.I. Jordan, L. El Ghaoui and S. Mian, \emph{Robust sparse hyperplane classifiers: Application to uncertain molecular profiling data},  Journal of Computational Biology, 11(6): 1073--1089, 2004.

	\bibitem{DucSin:09}
	J. Duchi and  Y. Singer, \textit{Efficient online and batch learning using forward backward splitting}, Journal of Machine Learning 	Research, 10: 2899--2934, 2009.
	
	\bibitem{GaiZho:18}
	B.R. Gaines, J. Kim and H. Zhou, \textit{Algorithms for fitting the constrained Lasso}, J. Comput. Graph. Stat., 27(4): 861--871, 2018.
	
	\bibitem{GarGow:23}
	G. Garrigos and R.M. Gower, \textit{Handbook of convergence theorems for (stochastic) gradient methods}, arXiv:2301.11235v2, 2023.
	
	\bibitem{GraBoy:13}
	M. Grant and S. Boyd, \textit{CVX: Matlab software for disciplined convex programming, version 2.0 beta}, http://cvxr.com/cvx, 2013.
	
	\bibitem{GorRic:20}
	E. Gorbunov, F. Hanzely and P. Richtarik, \textit{A unified theory of SGD: variance reduction, sampling, quantization and coordinate descent}, Proceedings of the 23rd International Conference on Artificial Intelligence and Statistics, vol. 108, 2020.
	
	\bibitem{HarSin:16}
	M. Hardt, B. Recht and Y. Singer, \textit{Train faster, generalize better: stability of stochastic gradient descent},  International Conference on Machine Learning, 2016.
	
	\bibitem{HuLin:15}
	Q. Hu, P. Zeng and L. Lin, \textit{The dual and degrees of freedom of linearly constrained generalized lasso}, Comput. Stat. Data Anal., 86:13--26, 2015.
	
	%\bibitem[Kundu et al., 2018]{KunBac:18}
	%A. Kundu, F. Bach and C. Bhattacharya, {\em Convex optimization over inter-section of simple sets: improved convergence rate guarantees via an exact penalty approach},  International Conference on Artificial Intelligence and Statistics, 2018.
	
	\bibitem{JamRus:19}
	G.M. James, C. Paulsonand and P. Rusmevichientong, \textit{Penalized and constrained optimization: an application to high-dimensional website advertising}, SIAM Journal on Optimization, 30(4), 3230--3251, 2019.
	
 \bibitem{JacObo:09}
L. Jacob,  G. Obozinski and J.P. Vert,  \textit{Group lasso with overlap and graph lasso}, International Conference on Machine Learning, 433–-440, 2009.
	
	\bibitem{JohZha:13}
	R. Johnson and T. Zhang, \emph{Accelerating stochastic gradient descent using predictive variance reduction}, Advances in Neural Information Processing Systems, 315--323, 2013.
	
	\bibitem{LewPan:98}
	A. Lewis and J.S. Pang, \emph{Error bounds for convex inequality systems}, Generalized Convexity, Generalized Monotonicity (J.-P. Crouzeix, J.-E.Martinez-Legaz, and M. Volle, eds.), 75--110, Cambridge University Press, 1998.

   \bibitem{LinMai:15}
  H. Lin,  J. Mairal and Z. Harchaoui, \emph{A universal catalyst for first-order optimization},   Advances in Neural Information Processing Systems Conference, 2015. 
	
	\bibitem{MouBac:11}
	E. Moulines and F. Bach, \textit{Non-asymptotic analysis of	 stochastic approximation algorithms for machine learning},  Advances in Neural Information Processing Systems Conf., 2011.
	
	\bibitem[Nedelcu et al., 2014]{NedNec:14}
	V. Nedelcu, I. Necoara and Q. Tran Dinh, \textit{Computational complexity of inexact gradient augmented Lagrangian methods:  application to constrained MPC}, SIAM Journal  on  Control and Optimization, 52(5): 3109--3134, 2014.
	
	\bibitem{NecSin:21}
	I. Necoara and N.K. Singh \textit{Stochastic subgradient projection methods for composite  optimization with  functional constraints}, arXiv preprint: 2204.08204, 2022.
	
	\bibitem{Nec:20}
	I. Necoara, \textit{General convergence analysis of stochastic first order methods for composite  optimization},  Journal of Optimization Theory and Applications, 189: 66--95  2021.
	
%	\bibitem{NecNes:15}
%	I. Necoara, Yu. Nesterov and F. Glineur, \textit{Linear convergence
%		of first  order methods for non-strongly convex optimization},
%	Mathematical Programming, 175(1): 69--107, 2019. 
	
	\bibitem{NemYud:83}
	A. Nemirovski and D.B. Yudin, \textit{Problem complexity and method efficiency in optimization}, Wiley Interscience, 1983.
	
	%\bibitem{NeeWar:16}
	 %D. Needell, N. Srebro and R. Ward, \textit{Stochastic gradient descent, weighted sampling, and the randomized Kaczmarz algorithm.} Mathematical Programming, 155: 549-–573, 2016.
	
	%\bibitem{17}
	%D. Needell, and R. Ward, \textit{ Batched stochastic gradient descent with
	%weighted sampling}, In Approximation Theory, Springer, volume 204 of Springer Proceedings in Mathematics and Statistics, 279--306, 2017.

    \bibitem{Nes:18}
	Yu. Nesterov, \textit{Lectures on Convex Optimization},  Springer Optimization and Its Applications, 137, 2018.
	
	\bibitem{AngNec:19}
	A. Nedich and I. Necoara. \emph{Random minibatch subgradient algorithms for
		convex problems with functional constraints}, Applied Mathematics and Optimization, 8(3): 801--833, 2019.
	
	\bibitem{NemJud:09}
	A. Nemirovski, A. Juditsky, G. Lan and A. Shapiro, \textit{Robust stochastic approximation approach to stochastic programming}, SIAM Journ.  Optimization, 19(4): 1574--1609, 2009.
	
	\bibitem{PenWan:19}
	X. Peng, L. Li and F. Wang, \textit{Accelerating minibatch stochastic gradient descent using typicality sampling}, IEEE Trans. Neural Networks Learn. Syst., 2019.
	
	\bibitem{Pol:67}
	B.T. Polyak, \emph{Minimization of unsmooth functionals}, USSR Computational Mathematics and Mathematical Physics, 9(3): 14-29, 1969.
	
	\bibitem{PolJud:92}
	B.T. Polyak and A.B. Juditsky, \emph{Acceleration of stochastic approximation by averaging}, SIAM Journal on Control and Optimization, 30(4): 838--855, 1992.
	
	%\bibitem{Pol:01}
	%B.T.  Polyak,  \emph{ Random algorithms for solving convex inequalities}, Studies in  Computational Mathematics, 8: 409--422,  2001.
	
	\bibitem{PatNec:17}
	A. Patrascu  and I. Necoara, \textit{Nonasymptotic convergence of stochastic proximal point algorithms for constrained convex optimization}, Journal of Machine Learning Research, 18(198): 1--42, 2018.
	
	\bibitem{RicTak:16}
	P. Richtarik and M. Takac, \emph{On optimal probabilities in stochastic coordinate descent methods}, Optimization Letters, 10(6): 1233-1243, 2016.
	
	\bibitem{RocUry:00}
	R.T.  Rockafellar and S.P. Uryasev, \emph{Optimization of conditional value-at-risk}, Journal of Risk,  2: 21--41, 2000.
	
	\bibitem{RobRic:19}
	R. Gower, L. Nicolas, Q. Xun, S. Alibek, S. Egor and P. Richtarik, \emph{SGD: General Analysis and Improved Rates},  International Conference on Machine Learning,  2019.
	
	\bibitem{RobMon:51}
	H. Robbins and S. Monro, \textit{A Stochastic Approximation Method}, The Annals of Mathematical Statistics, 22(3): 400–407, 1951.
	
	\bibitem{RosVil:14}
	L. Rosasco, S. Villa and B.C. Vu, \textit{Convergence of stochastic
		proximal  gradient algorithm},  Applied Mathematics and  Optimization, 82: 891--917 , 2020.
	
	\bibitem{RenZho:21}
	J. Renegar and S. Zhou \emph{A different perspective on the stochastic convex feasibility problem}, arXiv preprint: 2108.12029v1, 2021.
	
	\bibitem[Tibshirani, 2011]{Tib:11}
	R. Tibshirani,  \textit{The solution path of the generalized
		lasso}, Phd Thesis, Stanford Univ., 2011.
	
	\bibitem[Vapnik, 1998]{Vap:98}
	V. Vapnik, \textit{Statistical learning theory}, John Wiley, 1998.
	
	\bibitem{YanLin:16}
	T. Yang and Q. Lin, \textit{RSG: Beating subgradient method without smoothness and strong convexity},  Journal of Machine Learning Research, 19(6): 1--33, 2018.

    \bibitem{YinBuy:21}
    X. Yin and {\.I} B{\"u}y{\"u}ktahtak{\i}n,  \textit{A multi-stage stochastic programming approach to epidemic resource allocation with equity considerations}, Health Care Management Science, 24(3): 597--622, 2021.

    \bibitem{ZafVal:19}
    M. Zafar, I. Valera, M. Gomez-Rodriguez and K. Gummadi, \textit{Fairness constraints: A flexible approach for fair classification},  Journal of Machine Learning Research,
    20(1): 2737--2778, 2019.

	
\end{thebibliography}
\end{document}